\documentclass{article}
\usepackage{graphicx} 
\usepackage{amsmath,amssymb,amsthm}
\usepackage{tikz}
\usepackage{pgfplots}
\usepackage{booktabs}
\usepackage{tabularx}
\usepackage{multirow}
\usepackage{authblk}
\usepackage{cite}

\newtheorem{defin}{Definition}[section]
\newenvironment{definition}{\begin{defin}\rm}{\end{defin}}
\newtheorem{theorem}[defin]{Theorem}
\newtheorem{lemma}[defin]{Lemma}

\newtheorem{proposition}[defin]{Proposition}
\newtheorem{corollary}[defin]{Corollary}
\newtheorem{qu}[defin]{Question}
\newtheorem{exa}[defin]{Example}

\title{$n$-th Tropical Nevanlinna Theory}
\author[1]{Risto Korhonen}
\author[1,2]{Chengliang Tan\thanks{Corresponding author.}}
\affil[1]{Department of Physics and Mathematics, University of Eastern Finland, P.O. Box 111, FI-80101 Joensuu, Finland}
\affil[2]{Beijing International Center for Mathematical Research, Peking
University, 100871 Beijing, China.
}

\begin{document}
\date{}
\maketitle
\let\thefootnote\relax\footnotetext{E-mail address: risto.korhonen@uef.fi(R. Korhonen), chengliangtan@pku.edu.cn (C.L. Tan)}
\let\thefootnote\relax
\let\thefootnote\relax
\let\thefootnote\relax\footnotetext{The second author would like to thank the support of the China Scholarship Council (Grant No. 202206020078)}

\begin{abstract}
In this paper, the tropical Nevanlinna theory is extended for piecewise polynomial continuous functions. By constructing the $n$-th Poisson-Jensen formula, the $n$-th tropical counting, proximity, and characteristic functions are introduced, which have some different properties compared to the classical tropical setting. Then, not only is the $n$-th version of the second main theorem for tropical homogeneous polynomials obtained, but also a tropical second main theorem for ordinary Fermat type polynomials is acquired.
Moreover, by estimating the tropical logarithmic derivative with a growth assumption pointwise, a strong equality is proved.
This equality illustrates the relationship between $\sum_{i=0}^{m}N(r, 1_{0}\oslash f_{i})$ and the ramification term $N(r, C_{0}(f_{0}, \cdots, f_{m}))$, implying that there is no natural tropical truncated version of the second main theorem for shift operators.\end{abstract}
\baselineskip14pt

\section{Introduction}
Nevanlinna theory plays a very important role in complex analysis, in which the logarithmic derivative lemma plays a key role in the proof of the second main theorem. These theories have many applications in complex differential equations. Halburd and the first author first introduced the analogy of the theory for the difference operator \cite{Ndi}, proved the difference version of the second main theorem \cite{Dif}, obtained an analogue of the truncated second main theorem, and applied the theory to implement a complete classification of a class of second-order difference equations \cite{Sec}, considering the existence of finite order solutions as a good analogue of the Painlev\'e property \cite{Pai}. So far, a large number of studies on difference equations have been carried out \cite{Del, Fir}. Subsequently, Halburd, the first author and Tohge extended the theory to the complex projective space, proving the second fundamental theorem for hyperplanes, and also obtaining the difference version of the Picard theorem \cite{Hyp}.

Tropical mathematics operates within the tropical semiring, which consists of the real numbers (sometimes augmented with $-\infty$) equipped with the operations of maximum and summation, respectively. It has attracted a lot of research since the 1970s \cite{Min} and also has various applications such as statistical inference \cite{Sta}, group theory \cite{Gro}, supertropical forms \cite{Izh1, Izh2, Izh3}, and algebraic geometry \cite{Geo, Enu}. Moreover, tropical functions with shift operators have attracted considerable interest \cite{Upa}. In \cite{How}, a method is proposed to demonstrate that the integrability of the ultra-discrete equation is equivalent to the singularity confinement property of the difference equation. In \cite{Cel}, the detecting methods for integrability of ultra-discrete equations were discussed through singularity confinement.

In 2009, Halburd and Southall first introduced the tropical Nevanlinna theory \cite{Tnt}, and found that the existence of finite order max-plus meromorphic solutions is also a good detector of the integrability of the ultra-discrete equations. Building on this discovery, Laine and Tohge explored the tropical second main theorem and the tropical analogies of Valiron's and Clunie's Lemma \cite{Secm, Clu}. Subsequently, the first author and Tohge generalized the second main theorem to tropical projective space \cite{Pro}. In this paper, the authors discovered some apparent properties of the term $N(r, 1_{0}\oslash C_{0}(f_{0}, \cdots, f_{m}))$, intending to use this relationship as a breakthrough to obtain an analogue of the tropical truncated second main theorem. However, a corresponding proof cannot be provided at the moment. Similarly, Cao and Zheng proved the tropical second main theorem on tropical hypersurfaces and proposed two conjectures about the tropical truncated second main theorem based on the idea of the properties of the ramification term \cite{Thyp}. In \cite{Tnp}, Halonen, the first author and Filipuk removed the continuity in the classical continuous piecewise linear functions and generalized the definition of the corresponding characteristic functions. They then found that the growth assumption can be discarded, which could not be done in the classical meromorphic function domain, and finally considered the tropical inverse problem.

The main purpose of this paper is to generalize the tropical Nevanlinna theory from piecewise linear functions to piecewise polynomial functions. The paper is structured as follows: In the following section, we consider the ultra-discretization of rational functions and propose the concept of $n$-th tropical meromorphic functions. In the third section, we firstly prove the Poisson-Jensen formula (Theorem \ref{theorem3.3}) with respect to the $n$-th tropical meromorphic functions and then propose an analogy of the Characteristic function. Next, in section~4, we use a new method to obtain the pointwise estimation of the tropical difference quotient (Theorem \ref{theorem 4.6}) and deduce a tropical analogue of logarithmic derivative lemma for the polynomial version (Corollary \ref{cor4.7}). Moving on to section 5, the $n$-th version of the tropical second main theorem with tropical homogeneous polynomials is obtained (Theorem \ref{theorem5.6}), and since we have extended the linear functions to polynomials, we also obtain a tropical second main theorem with the ordinary Fermat polynomials (Theorem \ref{theorem5.8}). Finally, by applying the pointwise estimation of tropical logarithmic derivatives instead of the traditional tropical logarithmic derivative lemma, we prove the relationship (Theorem \ref{theorem6.2} and Corollary \ref{cor6.3}) proposed in \cite{Pro}. This implies that there is no natural analogue of the tropical truncated second main theorem concerning difference operators.

\section{$n$-th tropical meromorphic functions}
In the classical tropical algebra, we used to study the rational function
\begin{equation*}
    f(x)=\frac{a_{p}x^{p}+a_{n-1}x^{p-1}+\cdots+a_{0}}{b_{q}x^{q}+b_{q-1}x^{q-1}+\cdots+b_{0}}\not\equiv 0, \infty
\end{equation*}
where $p, q\in\mathbb{N}\cup\{0\}$ (we assume $\mathbb{N}$ does not include $0$), $x, a_{i}$ and $b_{j}$ are all non-negative for $i=0, \cdots p$ and $j=0, \cdots q$. After redefining $x=e^{\frac{X}{\epsilon}}$, $a_{i}=e^{\frac{A_{i}}{\epsilon}}$ and $b_{j}=e^{\frac{B_{j}}{\epsilon}}$, especially, $0=e^{\frac{-\infty}{\epsilon}}$, we can take a limits when $\epsilon\rightarrow 0^{+}$ where the process is called dequantization \cite{Upa, Cel} or ultradiscretization \cite{Tnt}, and then a tropical product is deduced by
\begin{equation*}
    F(X):=\lim\limits_{\epsilon\rightarrow 0^{+}}\epsilon\log f(X, \epsilon)=\max\limits_{0\leq i\leq p}\{iX+A_{i}\}-\max\limits_{0\leq j\leq q}\{jX+B_{j}\}
\end{equation*}
which is a piecewise linear continuous function.

In this paper, we are going to study the function $f(x)\not\equiv 0, \infty$ satisfying
\begin{equation*}
    \log f(x)=\prod_{k=1}^{n}\log \left(\frac{\sum_{i=0}^{p}a_{ki}x^{i}}{\sum_{j=0}^{q}b_{kj}x^{j}}\right)=:\prod_{k=1}^{n}\log f_{k}(x)
\end{equation*}
where $ n\in\mathbb{N}, p, q\in \mathbb{N}\cup\{0\}$ and we assume $x, a_{ki}$ and $b_{kj}$ are all non-negative for $k=1, \cdots, n$, $i=0, \cdots, p$ and $j=0, \cdots, q$. As usual, we define $x=e^{\frac{X}{\epsilon}}$, $a_{ki}=e^{\frac{A_{ki}}{\epsilon}}$ and $b_{kj}=e^{\frac{B_{kj}}{\epsilon}}$, and then define the tropical product by
\begin{eqnarray*}
    F(X)&:=&\lim\limits_{\epsilon\rightarrow 0^{+}}\epsilon^{n}\log f(X, \epsilon)=\prod_{k=1}^{n}\lim\limits_{\epsilon\rightarrow 0^{+}}\epsilon\log f_{k}(X, \epsilon)\\
    &=&\prod_{k=1}^{n}\left(\max\limits_{0\leq i\leq p}\{iX+A_{ki}\}-\max\limits_{0\leq j\leq q}\{jX+B_{kj}\}\right).
\end{eqnarray*}
As we can see, $F(X)$ is a piecewise polynomial continuous function. Now we determine the operations of this class of functions. We suppose that $F(X)$ and $G(X)$ are the tropical products of $f(x)$ and $g(x)$ separately, then the tropical products of $f(x)+g(x)$, $f(x)g(x)$ and $f(x)/g(x)$ are defined as
\begin{eqnarray*}
    F(X)\oplus G(X)&:=&\lim\limits_{\epsilon\rightarrow 0^{+}}\epsilon^{n}\log(f(X, \epsilon)+g(X, \epsilon))\\
    &=&\lim\limits_{\epsilon\rightarrow 0^{+}}\epsilon^{n}\left(\max\{\log f(X, \epsilon), \log g(X, \epsilon)\}+\log(1+h(X, \epsilon))\right)\\
    &=& \max\{F(X), G(X)\},\\
    F(X)\otimes G(X)&:=&\lim\limits_{\epsilon\rightarrow 0^{+}}\epsilon^{n}(\log f(X, \epsilon)+\log g(X, \epsilon))=F(X)+G(X),\\
    F(X)\oslash G(X)&:=&\lim\limits_{\epsilon\rightarrow 0^{+}}\epsilon^{n}(\log f(X, \epsilon)-\log g(X, \epsilon))=F(X)-G(X),
\end{eqnarray*}
where $0\leq h(X, \epsilon)=\frac{\min\{ f(X, \epsilon),  g(X, \epsilon)\}}{\max\{ f(X, \epsilon), g(X, \epsilon)\}} \leq 1$.  

To simplify the notations, we will adopt the notation $f(x)$ instead of $F(x)$ to express the piecewise polynomial functions from now on. Generally speaking, the continuous function $f(x)$ has the form
\begin{equation*}
    f(x)=f_{i}(x), \quad x\in [x_{i-1}, x_{i}],
\end{equation*}
where $f_{i}(x)$ is a polynomial with degree $n_{i}\in\mathbb{N}\cup\{0\}$, $i\in \mathbb{Z}$, and $\{x_{i}\}\subset\mathbb{R}$ is an increasing sequence. We also assume there is no finite accumulation point for $\{x_{i}\}$ and
\begin{eqnarray*}
    \sup\limits_{i}n_{i}&=:&n<\infty.
\end{eqnarray*}
In what follows in this paper, we refer to  functions $f(x)$ satisfying above conditions as the \textbf{$n-$th tropical meromorphic functions} (we also write $n$-th as 1-st, 2-nd, 3-rd, 4-th and so on for a specific $n$), and we always assume $n\in\mathbb{N}$ without explicitly stating it. This implies that we assume $f(x)$ is not a constant function unless  specified. In particular, $1$-st tropical meromorphic function is the classical non-constant object we study.

For each polynomial $f_{i}(x)=\sum_{j=0}^{n}a_{i,j}x^{j}$ and a given point $x_{0}\in \mathbb{R}$, we can Taylor expand $f_{i}(x)$ at $x_{0}$
\begin{equation*}
    f_{i}(x)-f_{i}(x_{0})=\sum_{j=1}^{n}\frac{f_{i}^{(j)}(x_{0})}{j!}(x-x_{0})^{j}
\end{equation*}
where
\begin{equation}\label{2a1}
    \frac{f_{i}^{(j)}(x_{0})}{j!}=\sum_{k=j}^{n}C_{k}^{j}a_{i, k}x_{0}^{k-j}
\end{equation}
and $C_{k}^{j}=\frac{k!}{j!(k-j)!}$ are the binomial coefficients. Hence
\begin{eqnarray*}
    \left(\begin{matrix}\frac{f_{i}^{(1)}(x_{0})}{1!} \\ 
    \frac{f_{i}^{(2)}(x_{0})}{2!}\\
    \vdots\\
    \frac{f_{i}^{(n)}(x_{0})}{n!}
    \end{matrix}\right)
    =\left(\begin{matrix}
    1 & C_{2}^{1}x_{0} & \cdots & C_{n}^{1}x_{0}^{n-1} \\ 
    0 & 1 & \cdots & C_{n}^{2}x_{0}^{n-2} \\
    \vdots & \vdots & \ddots &\vdots\\
    0 & 0 & \cdots & 1
    \end{matrix}\right)
    \left(\begin{matrix}a_{i,1} \\ 
    a_{i, 2}\\
    \vdots\\
    a_{i, n}
    \end{matrix}\right)=:\mathcal{C}_{n}(x_{0})\left(\begin{matrix}a_{i,1} \\ 
    a_{i, 2}\\
    \vdots\\
    a_{i, n}
    \end{matrix}\right),
\end{eqnarray*}
where the matrix $\mathcal{C}_{n}(x_{0})$ is not degenerate.

To define the poles, roots and their multiplicities of the $n$-th tropical meromorphic function $f(x)$, we have to introduce an auxiliary function
\begin{equation}\label{2a2}
    \omega_{f}^{(j)}(x)=\frac{sgn^{j+1}(x^{+})f^{(j)}(x^{+})-sgn^{j+1}(x^{-})f^{(j)}(x^{-})}{j!},
\end{equation}
where $1\leq j \leq n$ and $sgn(x)$ is the signum function. We remind here that $sgn(0^{\pm})=\pm1$. If $\omega_{f}^{(j)}(x)>0 $ ($<0$), then we call $x$ the \textbf{$j$-th root (pole)} of $f$ with multiplicity $|\omega_{f}^{(j)}(x)|$, and we refer to all $f$'s roots and poles as singularities of $f$. When $j=1$, it is nothing but the definition of roots, poles and their multiplicities of tropical meromorphic functions. Conventionally, the $j$-th poles or roots are those points where $f^{j}(x^{+})-f^{j}(x^{-})\neq 0$, but if $j$ is an even number, $x=0$ is a $j$-th root or pole that occurs when $f^{j}(x^{+})+f^{j}(x^{-})\neq 0$. With this definition, $x$ can be both the root and the pole of $f$ simultaneously, as illustrated in Table 1 in Example 3.4. This differs significantly from classical results. However, if we assume $x$ has $n$ clones with different multiplicities, then each clone can only be the root or pole of $f$ when it has a non-zero multiplicity.

Now, we define a new matrix by modifying $\mathcal{C}_{n}(x)$ slightly
\begin{eqnarray*}
    \mathcal{D}_{n}(x_{0}):=\left(\begin{matrix}
    1\times sgn^{2}(x_{0}) & C_{2}^{1}x_{0}\times sgn^{2}(x_{0}) & \cdots & C_{n}^{1}x_{0}^{n-1}\times sgn^{2}(x_{0}) \\ 
    0 & 1\times sgn^{3}(x_{0}) & \cdots & C_{n}^{2}x_{0}^{n-2}\times sgn^{3}(x_{0}) \\
    \vdots & \vdots & \ddots &\vdots\\
    0 & 0 & \cdots & 1\times sgn^{n+1}(x_{0})
    \end{matrix}\right).
\end{eqnarray*}
For any $x_{0}\in\mathbb{R}$, suppose $f(x)=\sum_{j=0}^{n}a_{j}x^{j}$ and $f(x)=\sum_{j=0}^{n}b_{j}x^{j}$ are the expressions on the right and left side of $x_{0}$, respectively. Then we obtain
\begin{eqnarray}\label{2a3}
     \left(\begin{matrix}\omega_{f}^{(1)}(x_{0}) \\ 
    \omega_{f}^{(2)}(x_{0})\\
    \vdots\\
    \omega_{f}^{(n)}(x_{0})
    \end{matrix}\right)
    =\mathcal{D}_{n}(x_{0}^{+})\left(\begin{matrix}a_{1} \\ 
    a_{2}\\
    \vdots\\
    a_{n}
    \end{matrix}\right)-\mathcal{D}_{n}(x_{0}^{-})\left(\begin{matrix}b_{1} \\ 
    b_{2}\\
    \vdots\\
    b_{n}
    \end{matrix}\right).
\end{eqnarray}

We say an $n$-th tropical meromorphic function $f$ is \textbf{$n$-th tropical entire} if it does not have any poles. Furthermore, if it has no roots at all, we call $f$ \textbf{$n$-th tropical nowhere vanishing entire}. It is worth noting that $f(x+c)$ and $f(x)\oplus g(x)$ may not be tropical entire for some $n$-th tropical entire functions $f(x)$ and $g(x)$ and $c\in \mathbb{R}$, which is not same as for classical tropical entire functions. For instance
\begin{exa}
    Define two tropical entire functions $f(x)=sgn(x)x^{2}$ and $g(x)=x$. Then
    \begin{eqnarray*}
        f(x)\oplus g(x)&=& \left\{\begin{array}{ll}
         x, & x\leq -1,\\
         -x^{2}, & -1< x \leq 0, \\
         x, & 0<x \leq 1, \\
         x^{2}, & 1 <x,
         \end{array}\right.\\
         f(x+1)&=&sgn(x+1)(x+1)^{2},
    \end{eqnarray*}
    from which we can see that $f(x)\oplus g(x)$ and $f(x+1)$ have $2$-nd poles with multiplicity $1$ and $2$ at $x=0$ and $x=-1$ respectively.
\end{exa}
In addition, $f(x)=sgn(x)x^{2}$ implies that $n$-th tropical entire functions may not be convex. However, they still possess some useful properties.
\begin{proposition}\label{p2.2}
    Let $n\in\mathbb{N}$, then for any $n$-th tropical entire function $f(x)$, we have
    \begin{equation}\label{2p2.2}
        \frac{f(r)+f(-r)}{2}\geq f(0)
    \end{equation}
    for all $r> 0$. In particular, the equality holds when $f(x)$ is tropical nowhere vanishing entire.
\end{proposition}
However, the opposite statement is not true. Indeed, let $f(x)=0$ when $|x|<1$, and $f(x)=-x+sgn(x)$ when $|x|>1$, then $f(x)$ is clearly not tropical entire, but $f(r)+f(-r)=2f(0)$ holds for all $r>0$.
\begin{proof}
    For $x>0$, we assume $f(x)=f_{i}(x), x\in[x_{i-1}, x_{i}]$, where $i=1, 2, \cdots,$ $\{x_{i}\}_{i>0}$ are all positive roots of $f$ and $x_{0}=0$. Then, for any $r>0$, we can find $i^{*}\in\mathbb{N}$ such that $r\in [x_{i^{*}-1}, x_{i^{*}})$, and then
    \begin{eqnarray}\label{2aa1}
        f(r)-f_{1}(r)&=&\sum_{i=1}^{i^{*}-1}(f_{i+1}(r)-f_{i}(r))\nonumber\\
        &=&\sum_{i=1}^{i^{*}-1}\sum_{j=1}^{n}\frac{f_{i+1}^{(j)}(x_{i})-f_{i}^{(j)}(x_{i})}{j!}(r-x_{i})^{j}\nonumber\\
        &=&\sum_{i=1}^{i^{*}-1}\sum_{j=1}^{n}\omega_{f}^{(j)}(x_{i})(r-|x_{i}|)^{j}\geq 0
    \end{eqnarray}
     where we Taylor expand $f_{i}(r)$ and $f_{i+1}(r)$ at $x_{i}$ for all $i=1, \cdots, i^{*}-1$ and use the fact $f_{i}(x_{i})=f_{i+1}(x_{i})$.
     
     For $x<0$, we assume $f(x)=f_{-i}(x), x\in[x_{-i}, x_{-i+1}]$, where $i=1, 2, \cdots,$ $\{x_{-i}\}_{i>0}$ are all negative roots of $f$ and $x_{0}=0$. Then for any $r>0$, we can also find $i_{*}\in\mathbb{N}$ such that $-r\in [x_{-i_{*}}, x_{-i_{*}+1})$, and then
    \begin{eqnarray}\label{2aa2}
        f(-r)-f_{-1}(-r)&=&\sum_{i=1}^{i_{*}-1}(f_{-i-1}(-r)-f_{-i}(-r))\nonumber\\
        &=&\sum_{i=1}^{i_{*}-1}\sum_{j=1}^{n}\frac{f_{-i-1}^{(j)}(x_{-i})-f_{-i}^{(j)}(x_{-i})}{j!}(-r-x_{-i})^{j}\nonumber\\
        &=&\sum_{i=1}^{i_{*}-1}\sum_{j=1}^{n}(-1)^{j}\frac{f_{-i-1}^{(j)}(x_{-i})-f_{-i}^{(j)}(x_{-i})}{j!}(r-|x_{-i}|)^{j}\nonumber\\
        &=&\sum_{i=1}^{i_{*}-1}\sum_{j=1}^{n}(-1)^{j+1}\frac{f_{-i}^{(j)}(x_{-i})-f_{-i-1}^{(j)}(x_{-i})}{j!}(r-|x_{-i}|)^{j}\nonumber\\
        &=&\sum_{i=1}^{i_{*}-1}\sum_{j=1}^{n}\omega_{f}^{(j)}(x_{-i})(r-|x_{-i}|)^{j}\geq 0.
    \end{eqnarray}
    Hence $f(r)+f(-r)\geq f_{1}(r)+f_{-1}(-r)$ for all $r>0$. On the other hand
    \begin{eqnarray}\label{2aa3}
        f_{1}(r)+f_{-1}(-r)&=&\sum_{j=0}^{n}\frac{f_{1}^{(j)}(0)}{j!}r^{j}+\sum_{j=0}^{n}\frac{f_{-1}^{(j)}(0)}{j!}(-r)^{j}\nonumber\\
        &=&2f(0)+\sum_{j=1}^{n}\frac{f_{1}^{(j)}(0)-(-1)^{j+1}f_{-1}^{(j)}(0)}{j!}r^{j}\nonumber\\
        &=&2f(0)+\sum_{j=1}^{n}\omega_{f}^{(j)}(0)r^{j}\nonumber\\
        &\geq& 2f(0).
    \end{eqnarray}
    It follows that (\ref{2p2.2}) holds. If $f(x)$ does not have any singularities, then $f(x)=f_{1}(x)$ for all $x>0$, $f(x)=f_{-1}(x)$ for all $x<0$ and $\omega_{f}^{(j)}(0)=0$ for $j=1, \cdots, n$, hence the assertion follows from last equality above.
\end{proof}

Furthermore, parallel to the result that all tropical meromorphic functions can be represented as a quotient of two tropical entire functions which do not have any common roots, we can get a similar proposition in our setting.

\begin{proposition}\label{2p2.3}
    Let $n\in\mathbb{N}$, then for any $n$-th tropical meromorphic function $f$, there exist $n_{1}$-th and $n_{2}$-th tropical entire functions $g$ and $h$ such that $f=h\oslash g$ $(\max\{n_{1}, n_{2}\}= n )$, where $h$ and $g$ do not have any common $j$-th roots and poles, $j=1, \cdots, n$.
\end{proposition}

\begin{proof}
    We will construct $h(x)$ and $g(x)$ separately to prove the assertion. We firstly consider $x>0$, and assume
    \begin{equation*}
    f(x)=f_{i}(x)=\sum_{j=0}^{n}a_{i, j}x^{j}, \quad x\in [x_{i-1}, x_{i}],
    \end{equation*}
    where $i=1, 2, \cdots$, $x_{0}=0$ and $\{x_{i}\}_{i>0}$ are all positive singularities of $f$, and all $f_{i}$ are polynomials with degree no more than $n$. We suppose
    \begin{equation}\label{2a31}
        g(x)=g_{i}(x)=\sum_{j=0}^{n}b_{i, j}x^{j}, x\in [x_{i-1}, x_{i}],
    \end{equation}
    then
    \begin{equation}\label{2a32}
        h(x)=f(x)+g(x)=\sum_{j=0}^{n}(a_{i, j}+b_{i, j})x^{j}, x\in [x_{i-1}, x_{i}],
    \end{equation}
    and it is clear that $\max\{n_{1}, n_{2}\}=n$. Now, we only need to determine the coefficients $\{b_{i, j}\}$, $j=0, \cdots, n$ and $i=1, 2, \cdots$. In fact, all these coefficients would be determined by $\{b_{1, j}\}_{j=0}^{n}$ which can be chosen arbitrarily. 
    
    Given arbitrary $\{b_{1, j}\}_{j=0}^{n}\subset\mathbb{R}$, if $\{b_{i, j}\}_{j=1}^{n}$ can be determined for all $i=2, 3, \cdots$, then $\{b_{i, 0}\}$ are also uniquely determined because of the continuity of $g(x)$. Hence, it is enough to find $\{b_{i, j}\}_{j=1}^{n}$. 
    
    Define
    \begin{eqnarray*}
        Q_{i}=\{q: x_{i}\text{ is a q-th root of } f\}\subset \{1, \cdots, n\}
    \end{eqnarray*}
    and
    \begin{eqnarray*}
        P_{i}=\{p: x_{i}\text{ is a p-th pole of } f\}\subset \{1, \cdots, n\},
    \end{eqnarray*} 
    thus $Q_{i}\cap P_{i}=\varnothing$ holds for all $i=1, 2, \cdots$. Let
    \begin{eqnarray}\label{2a33}
        \omega_{g}^{(j)}(x_{i})&=& \left\{\begin{array}{ll}
         -\omega_{f}^{(j)}(x_{i})>0, & j\in P_{i},\\
         0, & j\notin P_{i},
         \end{array}\right.
    \end{eqnarray}
    then
    \begin{eqnarray*}
        \omega_{h}^{(j)}(x_{i})=\omega_{f}^{(j)}(x_{i})+\omega_{g}^{(j)}(x_{i})=\left\{\begin{array}{ll}
         \omega_{f}^{(j)}(x_{i})>0, & j\in Q_{i},\\
         0, & j\notin Q_{i}.
         \end{array}\right.
    \end{eqnarray*}
     where $j=1, \cdots, n$ and $i=1, 2, \cdots$. Thus $x_{i}$ are neither the poles of $h$ and $g$, nor the common $j$-th roots of $h$ and $g$. 
     
     On the other hand, for abbreviation, let $\textbf{b}_{i}=(b_{i, 1}, \cdots, b_{i, n})^{T}$ and $\pmb{\omega}_{g}(x_{i})=(\omega_{g}^{(1)}(x_{i}), \cdots, \omega_{g}^{(n)}(x_{i}))^{T}$ where $\textbf{a}^{T}$ means the transpose of the vector $\textbf{a}$. Since $x_{i}$ is non-negative, then $\mathcal{D}(x_{i}^{\pm})=\mathcal{D}(x_{i})$. It follows from (\ref{2a3}) that
     \begin{eqnarray*}
         \pmb{\omega}_{g}(x_{i})=\mathcal{D}(x_{i})(\textbf{b}_{i+1}-\textbf{b}_{i})
     \end{eqnarray*}
     for all $i=1, 2, \cdots$. Since $\mathcal{D}_{n}(x_{i})$ is nondegenerate, then it follows by iteration that
     \begin{eqnarray}\label{2a34}
         \textbf{b}_{i+1}=\textbf{b}_{1}+\sum_{k=1}^{i}\mathcal{D}^{-1}(x_{k})\pmb{\omega}_{g}(x_{k}),
     \end{eqnarray}
     where $i=1, 2, \cdots$. Hence, by combining (\ref{2a31}), (\ref{2a32}), (\ref{2a33}) and (\ref{2a34}), we can see that $g(x)$ and $h(x)$ can be determined by $g_{1}$ when $x>0$. The conclusion that $g(x)$ and $h(x)$ can be determined by $g_{-1}$ when $x< 0$ can also be obtained by a similar analysis, where $g(x)=g_{-1}(x)=:\sum_{j=0}^{n}b_{-1, j}x^{j}$ when $x\in[x_{-1}, 0]$ and $x_{-1}$ is the largest negative singularity of $f(x)$. Moreover, $g_{-1}$ could also be determined by $g_{1}$ by a similar analysis as well. Specifically, we have $b_{-1, j}=(-1)^{j+1}(b_{1, j}-\omega_{g}^{(j)}(0))$ for $j=1, \cdots, n$ where $\omega_{g}^{(j)}(0)$ is defined similarly with (\ref{2a33}), and $b_{-1, 0}=b_{1, 0}$ because $g(x)$ should be continuous at $x=0$. Thus, we complete the proof.
\end{proof}

\section{$n$-th tropical Nevanlinna theory}
 The Poisson–Jensen formula is a cornerstone of the construction of Nevanlinna theory. Before we give the formula, we need to do some preparation.
 \begin{lemma}\label{lemma3.1}
     Let $n\in\mathbb{N}$ and $f$ be a $n$-th tropical meromorphic function on $[-r, r]$ for some $r>0$. For any $x\in (-r, r)$, we denote two increasing sequences $\{x_{i}\}_{i=1}^{\nu}\subset (x, r)$ and $\{x_{-i}\}_{i=1}^{\mu}\subset (-r, x)$ which consist of all the points such that $f^{(j)}(x_{\pm i})$ does not exist for some $j\in\{1, \cdots, n\}$. Then we have
     \begin{eqnarray}\label{3p1}
    f(r)-f(x)=\sum_{j=1}^{n}\left(\frac{f^{(j)}(x^{+})}{j!}(r-x)^{j}+\sum_{i=1}^{\nu}\tau_{f}^{(j)}(x_{i})(r-x_{i})^{j}\right)
\end{eqnarray}
and
\begin{eqnarray}\label{3p2}
    f(x)-f(-r)=\sum_{j=1}^{n}\left((-1)^{j+1}\frac{f^{(j)}(x^{-})}{j!}(r+x)^{j}+(-1)^{j}\sum_{i=1}^{\mu}\tau_{f}^{(j)}(x_{-i})(r+x_{-i})^{j}\right)\nonumber\\
\end{eqnarray}
where $\tau_{f}^{(j)}(x)=\frac{f^{(j)}(x^{+})-f^{(j)}(x^{-})}{j!}$. 
 \end{lemma}

\begin{proof}
We only prove (\ref{3p1}) here, (\ref{3p2}) can be proved similarly. Let $r=x_{\nu+1}$, $x=x_{0}$ and denote
\begin{equation*}
    f(x)=f_{i}(x), \quad x\in [x_{i-1}, x_{i}]
\end{equation*}
where $i=1, 2, \cdots \nu+1$ and $f_{i}(x)$ are all polynomials of degree no more than $n$. Then we have $f^{(j)}(x_{i}^{+})=f_{i+1}^{(j)}(x_{i})$ and $f^{(j)}(x_{i}^{-})=f_{i}^{(j)}(x_{i})$ for all $i=1, \cdots, \nu$ and $j=1, \cdots, n$. Since $f(x)$ is continuous, hence $f_{i+1}(x_{i})=f_{i}(x_{i})$ holds for all $i=1, \cdots, \nu$, then
\begin{eqnarray}\label{3a1}
    f(r)-f(x)&=&\sum_{i=0}^{\nu}\left(f_{i+1}(x_{i+1})-f_{i+1}(x_{i})\right)\nonumber\\
    &=&f_{1}(x_{1})-f_{1}(x)+\sum_{i=1}^{\nu}\left(f_{i+1}(x_{i+1})-f_{i+1}(x_{i})\right)\nonumber\\
    &=&f_{1}(x_{1})-f_{1}(x)+\sum_{i=1}^{\nu}[(f_{i+1}(r)-f_{i+1}(x_{i}))-(f_{i}(r)-f_{i}(x_{i}))]\nonumber\\
    &&+\sum_{i=1}^{\nu}[-(f_{i+1}(r)-f_{i+1}(x_{i+1}))+(f_{i}(r)-f_{i}(x_{i}))]\nonumber\\
    &=&f_{1}(r)-f_{1}(x)+\sum_{i=1}^{\nu}[(f_{i+1}(r)-f_{i+1}(x_{i}))-(f_{i}(r)-f_{i}(x_{i}))]\nonumber\\
    &=&\sum_{j=1}^{n}\left(\frac{f^{(j)}(x^{+})}{j!}(r-x)^{j}+\sum_{i=1}^{\nu}\frac{f_{i+1}^{(j)}(x_{i})-f_{i}^{(j)}(x_{i})}{j!}(r-x_{i})^{j}\right).\nonumber
\end{eqnarray}
In the third equality, we can define $A_{i}=f_{i}(r)-f_{i}(x_{i})$, then $\sum_{i=1}^{\nu}(-A_{i+1}+A_{i})=A_{1}$ because of $A_{\nu+1}=0$. Hence we complete the proof of (\ref{3p1}).
\end{proof}

\noindent\textbf{Remark}: It is easy to check that $\omega_{f}^{(j)}(x)=sgn^{j+1}(x)\tau_{f}^{(j)}(x)$ when $x\neq 0$. Furthermore, for any $x_{0}\in\mathbb{R}$, assume $f_{1}(x)$ and $f_{-1}(x)$ are the right and left polynomials of $x_{0}$ separately, we will adopt the notations $f_{x_{0}^{+}}(x)=f_{1}(x)$ and $f_{x_{0}^{-}}(x)=f_{-1}(x)$ where $x\in \mathbb{R}$. Then (\ref{3p1}) and (\ref{3p2}) can be rewritten as
\begin{eqnarray}\label{3p3}
    f(r)-f(x)=f_{x^{+}}(r)-f(x)+\sum_{j=1}^{n}\sum_{i=1}^{\nu}\omega_{f}^{(j)}(x_{i})(r-|x_{i}|)^{j}
\end{eqnarray}
when $x\geq 0$ and
\begin{eqnarray}\label{3p4}
    f(x)-f(-r)=f(x)-f_{x^{-}}(-r)-\sum_{j=1}^{n}\sum_{i=1}^{\mu}\omega_{f}^{(j)}(x_{-i})(r-|x_{-i}|)^{j}
\end{eqnarray}
when $x\leq 0$.

Next, we need to define some auxiliary functions and notations. In the following definitions, we assume $r>0$ and $x\in(-r, r)$.

In the proof of the classical Poisson-Jenson formula, two terms $r+x$ and $r-x$ are needed to take a difference, but it is not enough in the $n$-th setting. Hence we need a generalization. Suppose $k\in\mathbb{N}\cup\{0\}$, let
\begin{equation*}
    B_{k}(r, x):=\frac{(r+x)^{k}+(r-x)^{k}}{2}, \quad D_{k}(r, x):=\frac{(r+x)^{k}-(r-x)^{k}}{2}.
\end{equation*}
In particular, $B_{0}(r, x)=1$, $D_{0}(r, x)=D_{k}(r, 0)=0$, $B_{1}(r, x)=r$ and $D_{1}(r, x)=x$. Then
\begin{eqnarray*}
    B_{k}(r, x)+D_{k}(r, x)=(r+x)^{k}, \quad B_{k}(r, x)-D_{k}(r, x)=(r-x)^{k}
\end{eqnarray*}
are the generalizations of $r+x$ and $r-x$ respectively.

Define
\begin{eqnarray*}
    E(r, x, x_{i})&:=&r^{2}-|x-x_{i}|r-xx_{i}\nonumber\\
    &=&\left\{\begin{array}{ll}
         (r-x_{i})(r+x), & x\leq x_{i},\\
         (r+x_{i})(r-x), & x_{i}\leq x,
         \end{array}\right.
\end{eqnarray*}
where $x_{i}\in(-r, r)$, and, in particular, we have $E(r, 0, x_{i})=r(r-|x_{i}|)$ and $E(r, x, x)=r^{2}-x^{2}$.

Next we split $(-r, r)$ into nine disjoint parts 
\begin{eqnarray*}
    F_{1}(r, x)&:=&\{y\in (-r, r):-r<y<\min\{0, x\}\},\\
    F_{2}(r, x)&:=&\{y\in (-r, r):\max\{0, x\}< y<r\},\\
    F_{3}(r, x)&:=&\{y\in (-r, r):y=x=0\},\\
    F_{4}(r, x)&:=&\{y\in (-r, r):0< y<x\},\\
    F_{5}(r, x)&:=&\{y\in (-r, r):x< y<0\},\\
    F_{6}(r, x)&:=&\{y\in (-r, r):x< y=0\},\\
    F_{7}(r, x)&:=&\{y\in (-r, r):y=0<x\},\\
    F_{8}(r, x)&:=&\{y\in (-r, r):0<y=x\},\\
    F_{9}(r, x)&:=&\{y\in (-r, r):y=x<0\}.
\end{eqnarray*}
Some of these sets are empty for a fixed $x$. In particular, we have $F_{l}(r, 0)=\varnothing$ for $l=4, \cdots, 9$. From now on, without causing ambiguity, we will write $F_{l}(r, x)=F_{l}$.

For any $n$-th tropical meromorphic function $f$, we will find that following two functions will appear after multiplying (\ref{3p1}) and (\ref{3p2}) by $(r+x)^{n}$ and $(r-x)^{n}$ respectively, and then taking difference in the proof of Theorem \ref{theorem3.3}
\begin{eqnarray*}
    \Omega_{f}^{(j)}(x, x_{i}):=\frac{sgn^{j+1}(x_{i}-x)^{+}f^{(j)}(x_{i}^{+})-sgn^{j+1}(x_{i}-x)^{-}f^{(j)}(x_{i}^{-})}{j!},\\
    \Gamma_{f}^{(j)}(x, x_{i}):=\frac{sgn^{j}(x_{i}-x)^{+}f^{(j)}(x_{i}^{+})-sgn^{j}(x_{i}-x)^{-}f^{(j)}(x_{i}^{-})}{j!}.
\end{eqnarray*}
Especially, $\Omega_{f}^{(j)}(0, x_{i})=\omega_{f}^{(j)}(x_{i})$. As we can see, both of the functions are related to the multiplicity function $\omega_{f}^{(j)}(x_{i})$, and the following lemma shows the relationship.

\begin{lemma}\label{lemma3.2}
Let $n\in\mathbb{N}$ and $f$ be a $n$-th tropical meromorphic function, $r>0$ and $x, x_{i} \in(-r, r)$. Then for any $j=1, \cdots, n$, we have
    \begin{eqnarray*}
        \Omega_{f}^{(j)}(x, x_{i})&=& \omega_{f}(x_{i})\times\left\{\begin{array}{ll}
         1, & x_{i}\in F_{1}\cup F_{2} \cup F_{3} \cup F_{6} \cup F_{8},\\
         (-1)^{j+1}, & x_{i}\in F_{4}\cup F_{5} \cup F_{7} \cup F_{9}.
         \end{array}\right.\\
         &&+\frac{f^{(j)}(x_{i}^{-})}{j!}\times\left\{\begin{array}{ll}
         0, & x_{i}\in F_{1}\cup F_{2} \cup F_{3} \cup F_{4} \cup F_{5},\\
         (-1)^{j+1}-1, & x_{i}\in F_{6},\\
         1-(-1)^{j+1}, & x_{i}\in F_{7}\cup F_{8} \cup F_{9}.
         \end{array}\right.
    \end{eqnarray*}
    and
    \begin{eqnarray*}
        \Gamma_{f}^{(j)}(x, x_{i})&=& \omega_{f}(x_{i})\times\left\{\begin{array}{ll}
         -1, & x_{i}\in F_{1},\\
         1, & x_{i}\in F_{2}\cup F_{3} \cup F_{6} \cup F_{8},\\
         (-1)^{j}, & x_{i}\in F_{4}\cup F_{7},\\ (-1)^{j+1}, & x_{i}\in F_{5} \cup F_{9}.
         \end{array}\right.\\
         &&+\frac{f^{(j)}(x_{i}^{-})}{j!}\times\left\{\begin{array}{ll}
         0, & x_{i}\in F_{1}\cup F_{2} \cup F_{4} \cup F_{5},\\
         2\times(-1)^{j+1}, & x_{i}\in F_{3},\\
         (-1)^{j+1}-1, & x_{i}\in F_{6}\cup F_{7},\\
         1-(-1)^{j}, & x_{i}\in F_{8} \cup F_{9}.
         \end{array}\right.
    \end{eqnarray*}
\end{lemma}

\begin{proof}
    We only prove the case where $x_{i}\in F_{9}$, the other cases can be proved similarly. According to the definition, we obtain
    \begin{eqnarray*}
        &&\Omega_{f}^{(j)}(x, x_{i})=\frac{f^{(j)}(x_{i}^{+})-(-1)^{j+1}f^{(j)}(x_{i}^{-})}{j!}\\
        &=&(-1)^{j+1}\frac{(-1)^{j+1}f^{(j)}(x_{i}^{+})-(-1)^{j+1}f^{(j)}(x_{i}^{-})+(-1)^{j+1}f^{(j)}(x_{i}^{-})-f^{(j)}(x_{i}^{-})}{j!}\\
        &=&(-1)^{j+1}\frac{sgn^{j+1}(x_{i}^{+})f^{(j)}(x_{i}^{+})-sgn^{j+1}(x_{i}^{-})f^{(j)}(x_{i}^{-})}{j!}+(1-(-1)^{j+1})\frac{f^{(j)}(x_{i}^{-})}{j!}\\
        &=&(-1)^{j+1}\omega_{f}^{(j)}(x_{i})+(1-(-1)^{j+1})\frac{f^{(j)}(x_{i}^{-})}{j!}
    \end{eqnarray*}
    and
    \begin{eqnarray*}
        &&\Gamma_{f}^{(j)}(x, x_{i})=\frac{f^{(j)}(x_{i}^{+})-(-1)^{j}f^{(j)}(x_{i}^{-})}{j!}\\
        &=&(-1)^{j+1}\frac{(-1)^{j+1}f^{(j)}(x_{i}^{+})-(-1)^{j+1}f^{(j)}(x_{i}^{-})+(-1)^{j+1}f^{(j)}(x_{i}^{-})+f^{(j)}(x_{i}^{-})}{j!}\\
        &=&(-1)^{j+1}\frac{sgn^{j+1}(x_{i}^{+})f^{(j)}(x_{i}^{+})-sgn^{j+1}(x_{i}^{-})f^{(j)}(x_{i}^{-})}{j!}+(1-(-1)^{j})\frac{f^{(j)}(x_{i}^{-})}{j!}\\
        &=&(-1)^{j+1}\omega_{f}^{(j)}(x_{i})+(1-(-1)^{j})\frac{f^{(j)}(x_{i}^{-})}{j!}.
    \end{eqnarray*}
\end{proof}

Now we are ready to give the cornerstone formula.

\begin{theorem}\label{theorem3.3}
    (\textbf{$n$-th tropical Poisson-Jensen formula}) Let $n\in\mathbb{N}$ and $f$ be a $n$-th tropical meromorphic function on $[-r, r]$ for some $r>0$. For any $x\in (-r, r)$, we denote a sequence $\{x_{i}\}$ consisting of all $f$'s singularities in $(-r, r)$, and denote by $y_{i, j}=x_{i}$ and $z_{i, j}=x_{i}$ the $j$-th roots and poles of $f$ respectively. Then we have
    \begin{eqnarray}\label{3pj}
    &&f(x)=\frac{1}{2}[f(r)+f(-r)]+\frac{D_{n}(r, x)}{2B_{n}(r, x)}[f(r)-f(-r)]\nonumber\\
    &-&\sum_{j=1}^{n}\left(\sum_{x_{i}\in F_{1}\cup F_{2}\cup F_{3}\cup F_{6}\cup F_{8}}+(-1)^{j+1}\sum_{x_{i}\in F_{4}\cup F_{5}\cup F_{7}\cup F_{9}}\right)\nonumber\\
    &&\quad \times\omega_{f}^{(j)}(x_{i})E^{j}(r, x, x_{i})\frac{B_{n-j}(r, x)}{2B_{n}(r, x)}\nonumber\\
    &-&\sum_{j=1}^{n}\left(-\sum_{x_{i}\in F_{1}}+\sum_{x_{i}\in F_{2} \cup F_{6} \cup F_{8}}+(-1)^{j}\sum_{x_{i}\in F_{4}\cup F_{7}}+(-1)^{j+1}\sum_{x_{i}\in F_{5}\cup F_{9}}\right)\nonumber\\
    &&\quad \times\omega_{f}^{(j)}(x_{i})E^{j}(r, x, x_{i})\frac{D_{n-j}(r, x)}{2B_{n}(r, x)}\nonumber\\
    &-&\left(\frac{(r+x)^{n}}{2B_{n}(r, x)}(f_{x^{-}}(r)-f(x))+\frac{(r-x)^{n}}{2B_{n}(r, x)}(f_{x^{-}}(-r)-f(x))\right)\nonumber\\
    &-&\left(sgn(x^{-})\frac{(r-|x|)^{n}}{2B_{n}(r, x)}(f_{0^{-}}(r)+f_{0^{-}}(-r)-2f(0))\right).
\end{eqnarray}
  In particular, when $x=0$, we have the 
$n$-th tropical Jensen Formula
\begin{eqnarray}\label{3j}
    f(0)&=&\frac{1}{2}[f(r)+f(-r)]\nonumber\\
    &&-\frac{1}{2}\sum_{j=1}^{n}\sum_{|y_{i, j}|<r}|\omega_{f}^{(j)}(y_{i, j})|(r-|y_{i,j}|)^{j}\nonumber\\
    &&+\frac{1}{2}\sum_{j=1}^{n}\sum_{|z_{i, j}|<r}|\omega_{f}^{(j)}(z_{i, j})|(r-|z_{i, j}|)^{j}.
\end{eqnarray}
\end{theorem}
 To better understand the $n$-th tropical Jensen formula, we give an example below.
 \begin{exa}
     Denote a $3$-rd tropical meromorphic function $f$ in $(-3, 3)$ by
     \begin{eqnarray*}
        f(x)&=& \left\{\begin{array}{ll}
         -2x^{3}-17x^{2}-46x-40, & -3<x\leq -2,\\
         2x^{2}+6x+4, & -2<x\leq -1,\\
         x^{2}+2x+1, & -1<x\leq 1,\\
         -x+5, & 1<x\leq 2,\\
         3x^{3}-15x^{2}+18x+3, & 2<x< 3,\\
         \end{array}\right.
    \end{eqnarray*}
    $f$'s singularities in $(-3, 3)$ are shown in the Table 1 where the blanks are $0$.
    \end{exa}

\begin{table}[ht]
\caption{All $f$'s singularities in $(-3, 3)$}\label{tab2}
\begin{tabular*}{\textwidth}{@{\extracolsep\fill}lcccccc}
\toprule%
$|\omega_{f}^{(j)}(x)|$ & $-2$ & $-1$ & $0$ & $1$ & $2$ \\
\midrule
$1$-st root  &  &  &  &  &  \\
$2$-nd root  &  & $1$  & $2$  &  & $3$ \\
$3$-rd root  & $2$ &  &  &  & $3$ \\
$1$-st pole  & & $2$ &  & $5$ & $5$\\
$2$-nd pole  & $7$ &  & & $1$ & \\
$3$-rd pole  &  &  &  &  & \\
\toprule
\end{tabular*}
\end{table}
Then we have the $3$-rd tropical Jensen formula when $2< r< 3$,
    \begin{eqnarray*}
        &&\frac{1}{2}((3r^{3}-15r^{2}+18r+3)+(2r^{3}-17r^{2}+46r-40))\\
        &-&\frac{1}{2}((r-1)^{2}+2r^{2}+3(r-2)^{2}+2(r-2)^{3}+3(r-2)^{3})\\
        &+&\frac{1}{2}(2(r-1)+5(r-1)+5(r-2)+7(r-2)^{2}+(r-1)^{2})\\
        &=&1=f(0).
    \end{eqnarray*}

\begin{proof}[Proof of Theorem \ref{theorem3.3}]
Since $r$ and $x$ are fixed, we will simplify $B_{k}(r, x)$, $D_{k}(r, x)$ and $E(r, x, x_{i})$ as $B_{k}$, $D_{k}$ and $E(x_{i})$ for all $k=0, \cdots, n$ in the following proof. It is worth noting that $x=0$ may be the singularity but $f^{(j)}(0)$ exits for all $j=1, \cdots, n$ and vise versa. Hence, to apply Lemma \ref{lemma3.1}, we need to redefine $\{x_{i}\}\cup\{x\}\cup\{0\}=:\{x_{i}\}_{i=-\mu}^{\nu}$ list according to their value and $x_{0}=x$, that is, $x_{-\mu}<x_{-\mu+1}<\cdots<x_{0}=x<x_{1}<\cdots<x_{\nu}$. Thus $\{x_{i}\}_{i=1}^{\nu}$ and $\{x_{-i}\}_{i=1}^{\mu}$ include all points that $f^{(j)}(x)$ does not exist for some $j=1, \cdots, n$ in $(x, r)$ and $(-r, x)$ separately. Then we can multiply (\ref{3p1}) and (\ref{3p2}) by $(r+x)^{n}$ and $(r-x)^{n}$ respectively, and then take a difference,
\begin{eqnarray*}
    &&(f(r)-f(x))(r+x)^{n}-(f(x)-f(-r))(r-x)^{n}\\
    &=&\sum_{j=1}^{n}\Bigg(\frac{f^{(j)}(x^{+})}{j!}(r^{2}-x^{2})^{j}(B_{n-j}+D_{n-j})\\
    &&\quad -(-1)^{j+1}\frac{f^{(j)}(x^{-})}{j!}(r^{2}-x^{2})^{j}(B_{n-j}-D_{n-j})\Bigg)\\
    &+&\sum_{j=1}^{n}\Bigg(\sum_{i=1}^{\nu}\tau_{f}^{(j)}(x_{i})E^{j}(x_{i})(B_{n-j}+D_{n-j})\\
    &&\quad -(-1)^{j}\sum_{i=1}^{\mu}\tau_{f}^{(j)}(x_{-i})E^{j}(x_{-i})(B_{n-j}-D_{n-j})\Bigg)
\end{eqnarray*}
\begin{eqnarray*}
    &=&\sum_{j=1}^{n}\left(\frac{f^{(j)}(x^{+})-(-1)^{j+1}f^{(j)}(x^{-})}{j!}(r^{2}-x^{2})^{j}B_{n-j}\right)\\
    &+&\sum_{j=1}^{n}\left(\frac{f^{(j)}(x^{+})-(-1)^{j}f^{(j)}(x^{-})}{j!}(r^{2}-x^{2})^{j}D_{n-j}\right)\\
    &+&\sum_{j=1}^{n}\sum_{i=-\mu, i\neq 0}^{\nu}\left(sgn^{j+1}(x_{i}-x)\tau_{f}^{(j)}(x_{i})E^{j}(x_{i})B_{n-j}\right)\\
    &+&\sum_{j=1}^{n}\sum_{i=-\mu, i\neq 0}^{\nu}\left(sgn^{j}(x_{i}-x)\tau_{f}^{(j)}(x_{i})E^{j}(x_{i})D_{n-j}\right)\\
    &=&\sum_{j=1}^{n}\sum_{i=-\mu}^{\nu}\left(\Omega_{f}^{(j)}(x, x_{i})E^{j}(x_{i})B_{n-j}+\Gamma_{f}^{(j)}(x, x_{i})E^{j}(x_{i})D_{n-j}\right)\\
    &=&\sum_{j=1}^{n}\left(\sum_{x_{i}\in F_{1}\cup F_{2}\cup F_{3}\cup F_{6}\cup F_{8}}+(-1)^{j+1}\sum_{x_{i}\in F_{4}\cup F_{5}\cup F_{7}\cup F_{9}}\right)\omega_{f}^{(j)}(x_{i})E^{j}(x_{i})B_{n-j}\\
    &+&\sum_{j=1}^{n}\left(((-1)^{j+1}-1)\sum_{x_{i}\in F_{6}}+(1-(-1)^{j+1})\sum_{x_{i}\in F_{7}\cup F_{8}\cup F_{9}}\right)\frac{f^{(j)}(x_{i}^{-})}{j!}E^{j}(x_{i})B_{n-j}\\
    &+&\sum_{j=1}^{n}\left(-\sum_{x_{i}\in F_{1}}+\sum_{x_{i}\in F_{2} \cup F_{6} \cup F_{8}}+(-1)^{j}\sum_{x_{i}\in F_{4}\cup F_{7}}+(-1)^{j+1}\sum_{x_{i}\in F_{5}\cup F_{9}}\right)\\
    &&\quad \times \omega_{f}^{(j)}(x_{i})E^{j}(x_{i})D_{n-j}\\
    &+&\sum_{j=1}^{n}\left(((-1)^{j+1}-1)\sum_{x_{i}\in F_{6}\cup F_{7}}+(1-(-1)^{j})\sum_{x_{i}\in F_{8}\cup F_{9}}\right)\frac{f^{(j)}(x_{i}^{-})}{j!}E^{j}(x_{i})D_{n-j}.
\end{eqnarray*}
The last equality follows by Lemma \ref{lemma3.2} and $D_{n-j}=0$ when $x_{i}\in F_{3}$ for all $j=1, \cdots, n$.
Now we combine the second and the fourth terms in the last equality. If $x=0$, then $F_{6}=F_{7}=F_{8}=F_{9}=\varnothing$, hence the combination would be $0$. If $x\neq 0$, and then
\begin{eqnarray*}
    &&\sum_{j=1}^{n}((-1)^{j+1}-1)\sum_{x_{i}\in F_{6}}\frac{f^{(j)}(x_{i}^{-})}{j!}E^{j}(x_{i})(r+x)^{n-j}\\
    &+&\sum_{j=1}^{n}(1-(-1)^{j+1})\sum_{x_{i}\in F_{7}}\frac{f^{(j)}(x_{i}^{-})}{j!}E^{j}(x_{i})(r-x)^{n-j}\\
    &+&\sum_{j=1}^{n}\sum_{x_{i}\in F_{8}\cup F_{9}}\frac{f^{(j)}(x_{i}^{-})}{j!}E^{j}(x_{i})(r+x)^{n-j}\\
    &+&\sum_{j=1}^{n}(-1)^{j}\sum_{x_{i}\in F_{8}\cup F_{9}}\frac{f^{(j)}(x_{i}^{-})}{j!}E^{j}(x_{i})(r-x)^{n-j}
\end{eqnarray*}
\begin{eqnarray*}
    &=&sgn(x^{-})(r-|x|)^{n}\sum_{j=1}^{n}(1-(-1)^{j+1})\frac{f^{(j)}(0^{-})}{j!}r^{j}\\
    &+&(r+x)^{n}\sum_{j=1}^{n}\frac{f^{(j)}(x^{-})}{j!}(r-x)^{j}\\
    &+&(r-x)^{n}\sum_{j=1}^{n}(-1)^{j}\frac{f^{(j)}(x^{-})}{j!}(r+x)^{j}\\
    &=&sgn(x^{-})(r-|x|)^{n}(f_{0^{-}}(r)+f_{0^{-}}(-r)-2f(0))\\
    &+&(r+x)^{n}(f_{x^{-}}(r)-f(x))+(r-x)^{n}(f_{x^{-}}(-r)-f(x)).
\end{eqnarray*}
In fact, we can check that the above combination is still equal to $0$ when we substitute $x=0$ into the last equality above. Hence, we don't need to distinguish these two cases. By a simplification, we obtain the form (\ref{3pj}). Although, $\{x_{i}\}_{i=-\mu}^{\nu}$ may include some non-singularities $x_{i}$ of $f$, $\omega_{f}^{(j)}(x_{i})=0$ at the moment, hence we can recover $\{x_{i}\}$ to be all $f$'s singularities in $(-r, r)$. 

When $x=0$, we have $F_{l}=\varnothing$ for $l=4, \cdots, 9$, $D_{n-j}=0$, $B_{n-j}=r^{n-j}$ for $j=1, \cdots, n$, and $E(x_{i})=r(r-|x_{i}|)$, then we have
\begin{eqnarray*}
    f(0)&=&\frac{1}{2}[f(r)+f(-r)]-\frac{1}{2}\sum_{j=1}^{n}\sum_{|x_{i}|<r}\omega_{f}^{(j)}(x_{i})(r-|x_{i}|)^{j}.
\end{eqnarray*}
By dividing $\{x_{i}\}$ to be the $j$-th poles and roots, we can get the $n$-th Jensen formula (\ref{3j}).
\end{proof}

In fact, the $n$-th Jensen formula can be easily obtained by adding the equations of (\ref{2aa1}), (\ref{2aa2}) and (\ref{2aa3}) in the proof of Proposition \ref{p2.2}. From the $n$-th Jensen formula, we can define the analogue of tropical characteristic functions.

We firstly define the \textbf{$n$-th max-plus proximity function} as usual
    \begin{equation*}
        m(r, f)=\frac{f^{+}(r)+f^{+}(-r)}{2}
    \end{equation*}
where $f^{+}(x):=\max\{f(x), 0\}$, and then denote the \textbf{$j-$th integrated max-plus counting function} by
    \begin{eqnarray}\label{3aN}
        N^{(j)}(r, f)&:=&\frac{1}{2}\sum_{i=1}^{\mu_{j}}|\omega_{f}^{(j)}(z_{i, j})|(r-|z_{i, j}|)^{j}\nonumber\\
        &=&\frac{1}{2}\int_{0}^{r}\cdots\int_{0}^{r}n^{(j)}\left(\min_{1\leq s \leq j}\{t_{s}\}, f\right)dt_{1}\cdots dt_{j}
    \end{eqnarray}
where $\{z_{i, j}\}$ are all $f$'s $j$-th poles in the interval $(-r, r)$ and \textbf{$j$-th max-plus counting function} $n^{(j)}(r, f)$ gives the number of $\{z_{i, j}\}$, counting multiplicities. The second equality follows by
    \begin{eqnarray*}
        &&\int_{0}^{r}\cdots\int_{0}^{r}n^{(j)}\left(\min_{1\leq s \leq j}\{t_{s}\}, f\right)dt_{1}\cdots dt_{j}\\
        &=&\sum_{k=1}^{\mu_{j}}n^{(j)}(|z_{k+1, j}|, f)\left((r-|z_{k, j}|)^{j}-(r-|z_{k+1, j}|)^{j}\right)\\
        &=&\sum_{k=1}^{\mu_{j}}\left(\sum_{i=1}^{k}|\omega_{f}^{(j)}(z_{i, j})|\right)\left((r-|z_{k, j}|)^{j}-(r-|z_{k+1, j}|)^{j}\right)\\
        &=&\sum_{i=1}^{\mu_{j}}|\omega_{f}^{(j)}(z_{i, j})|\sum_{k=i}^{\mu_{j}}\left((r-|z_{k, j}|)^{j}-(r-|z_{k+1, j}|)^{j}\right)\\
        &=&\sum_{i=1}^{\mu_{j}}|\omega_{f}^{(j)}(z_{i, j})|(r-|z_{i, j}|)^{j},
    \end{eqnarray*}
    where we assume $\{z_{i, j}\}_{i=1}^{\mu_{j}}$ are all $f$'s $j$-th poles in $(-r, r)$, $z_{\mu_{j}+1, j}=r$ and $z_{0, j}=0$ for all $j=1, \cdots, n$. In the third equality, we used a formula that double summation exchange the order, that is, $\sum_{j=1}^{n}\sum_{k=j}^{n}a_{k}b_{j,k}=\sum_{k=1}^{n}a_{k}\sum_{j=1}^{k}b_{j,k}$ for some constants $a_{k}, b_{j, k}\in\mathbb{R}$. 
    
    Note that, in classical setting, $N^{(j)}(r, f)$ or $N^{j)}(r, f)$ is usually denoted as the integrated counting function of poles with the multiplicities truncated by $j$, and then we can have the truncated second main theorem with this notation. However, we know that there may not be natural tropical version of truncated second main theorem from Section 6, hence we are free to adopt the notation $(\ref{3aN})$.

In addition, we need to enhance the definition of counting function $N^{(j)}(r, f)$ properly. For any given three constants $r_{1}$, $r_{2}$ and $r>0$ satisfying $-r\leq r_{1}\leq r_{2}\leq r$, we suppose $\{z_{i, j}\}_{i=1}^{\mu_{j}}$ are all $f'$s $j$-th poles in $(r_{1}, r_{2})$ listed by their value with multiplicities $|\omega_{f}^{(j)}(z_{i, j})|$. Then the \textbf{partial $j-$th integrated max-plus counting function}
     $N_{(r_{1}, r_{2})}^{(j)}(r, f)$ is defined by
    \begin{eqnarray*}
        N_{(r_{1}, r_{2})}^{(j)}(r, f)&:=&\frac{1}{2}\sum_{i=1}^{\mu_{j}}|\omega_{f}^{(j)}(z_{i, j})|(r-|z_{i, j}|)^{j}
    \end{eqnarray*}
for $j=1, \cdots n$.
Especially, $N_{(-r, r)}^{(j)}(r, f)=N^{(j)}(r, f)$. Similarly, $N_{(r_{1}, r_{2})}^{(j)}(r, -f)$, $N_{[r_{1}, r_{2})}^{(j)}(r, \pm f)$, $N_{(r_{1}, r_{2}]}^{(j)}(r, \pm f)$ and $N_{[r_{1}, r_{2}]}^{(j)}(r, \pm f)$ can be defined.
    
Then the \textbf{$n$-th max-plus characteristic function} is denoted by
    \begin{equation*}
        T(r, f)=m(r, f)+\sum_{j=1}^{n}N^{(j)}(r, f).
    \end{equation*}
With these notations, we can change $n$-th Jensen formula as
    \begin{theorem} \textbf{($n$-th Jensen formula)}
        Let $n\in\mathbb{N}$ and $f$ be a $n$-th tropical meromorphic function in $\mathbb{R}$, then
        \begin{equation}\label{3j1}
        T(r, f)=T\left(r, \frac{1_{0}}{f}\oslash\right)+f(0)
    \end{equation}
    for any $r>0$.
    \end{theorem}

    We also define the order and hyper-order of $f$ by
    \begin{eqnarray*}
        \rho(f)&:=&\limsup\limits_{r\rightarrow \infty}\frac{\log T(r, f)}{\log r},\\
        \rho_{2}(f)&:=&\limsup\limits_{r\rightarrow \infty}\frac{\log\log T(r, f)}{\log r}.
    \end{eqnarray*}

In the classical tropical setting, Laine and Tohge defined the so-called tropical hyper-exponential function according to the definition of hyper-order \cite{Secm}.
\begin{definition}
    \cite{Secm} Let $\alpha$ be a real number with $|\alpha|>1$. Define a function $e_{\alpha}(x)$ on $\mathbb{R}$ by
    \begin{eqnarray*}
        e_{\alpha}(x):=\alpha^{\lfloor x\rfloor}(x-\lfloor x\rfloor)+\sum_{i=-\infty}^{\lfloor x\rfloor-1}\alpha^{i}=\alpha^{\lfloor x\rfloor}\left(x-\lfloor x\rfloor+\frac{1}{\alpha-1}\right)
    \end{eqnarray*}
    where $\lfloor x\rfloor$ means the integer part of $x$.
\end{definition}

Then, they proved $\overline{e}_{\alpha}(x)=\alpha e_{\alpha}(x)$ \cite[Proposition 8.3 ]{Secm} and $\rho_{2}(e_{\alpha}(x))=1$, \cite[Proposition 8.5]{Secm}. Similarly, we can define a special $n$-th tropical meromorphic function whose hyper order is 1 and which is a simple generalization of $e_{\alpha}(x)$.
\begin{definition}
    Let $\alpha$ be a real number with $\alpha>1$ and $n\in\mathbb{N}$. Define a function $e_{n, \alpha}(x)$ on $\mathbb{R}$ by
    \begin{eqnarray*}
        e_{n, \alpha}(x)&:=&sgn^{n+1}(x^{+})\alpha^{\lfloor x\rfloor}(x^{n}-\lfloor x\rfloor^{n})+\sum_{i=-\infty}^{\lfloor x\rfloor-1}sgn^{n+1}(i^{+})\alpha^{i}((i+1)^{n}-i^{n}).
    \end{eqnarray*}
\end{definition}

 Clearly, $e_{1, \alpha}(x)=e_{\alpha}(x)$. Besides, $e_{n, \alpha}(x)$ also has similar properties with $e_{\alpha}(x)$, as shown in the following proposition.
\begin{proposition}\label{pro3.8}
    $e_{n,\alpha}(x)$ has the following properties:

    $(i)$ $e_{n,\alpha}(x)$ is a non-negative non-decreasing $n$-th tropical entire function,

    $(ii)$ $\rho_{2}(e_{n,\alpha})$=1.
\end{proposition}
\begin{proof}
    $(i)$ From the definition, we can see that the potential points of discontinuity are $m\in \mathbb{Z}$, while
    \begin{eqnarray*}
        \lim\limits_{x\rightarrow m^{-}}e_{n, \alpha}(x)&=&sgn^{n+1}(m^{-})\alpha^{m-1}(m^{n}-(m-1)^{n})\\
        &&+\sum_{i=-\infty}^{m-2}sgn^{n+1}(i^{+})\alpha^{i}((i+1)^{n}-i^{n})\\
        &=&\sum_{i=-\infty}^{m-1}sgn^{n+1}(i^{+})\alpha^{i}((i+1)^{n}-i^{n})\\
        &=&\lim\limits_{x\rightarrow m^{+}}e_{n, \alpha}(x).
    \end{eqnarray*}
    Thus $e_{n, \alpha}(x)$ is continuous on $\mathbb{R}$. Furthermore, for any $x_{0}\in \mathbb{R}$, since $e_{n, \alpha}^{'}(x_{0}^{+})=\lim\limits_{x\rightarrow x_{0}^{+}}sgn^{n+1}(x^{+})\alpha^{\lfloor x\rfloor}n x^{n-1}=\lim\limits_{x\rightarrow x_{0}^{+}}\alpha^{\lfloor x\rfloor}|x|^{n-1}\geq 0$, then it is non-decreasing. In addition,
    \begin{eqnarray*}
        \lim\limits_{x\rightarrow -\infty}e_{n, \alpha}(x)&=&\lim\limits_{m\in\mathbb{Z}, m\rightarrow -\infty}e_{n, \alpha}(m)\\
        &=&\lim\limits_{m\in\mathbb{Z}, m\rightarrow -\infty}\sum_{i=-\infty}^{m-1}sgn^{n+1}(i^{+})\alpha^{i}((i+1)^{n}-i^{n})\\
        &=&\lim\limits_{-m\in\mathbb{N}, m\rightarrow -\infty}\sum_{i=-\infty}^{m-1}\alpha^{i}(|i|^{n}-|i+1|^{n})\\
        &\geq& 0,
    \end{eqnarray*}
    which implies that it is non-negative. Besides,
    \begin{eqnarray}
        \omega_{e_{n, \alpha}}^{(j)}(m)&=&sgn^{j+1}(m^{+})C_{n}^{j}sgn^{n+1}(m^{+})\alpha^{m}m^{n-j}\nonumber\\
        &&-sgn^{j+1}(m^{-})C_{n}^{j}sgn^{n+1}(m^{-})\alpha^{m-1}m^{n-j}\nonumber\\
        &=&C_{n}^{j}|m|^{n-j}\alpha^{m-1}(\alpha-1)\nonumber\\
        &\geq& 0\nonumber,
    \end{eqnarray}
    where we assume $0^{0}=1$ here and $j=1, \cdots, n$. Hence the assertion follows.

    $(ii)$ When $x=0$, we have
    \begin{eqnarray*}
        |e_{n, \alpha}(0)|&=&\sum_{i=-\infty}^{-1}(-1)^{n+1}\alpha^{i}((i+1)^{n}-i^{n})\\
        &=&\sum_{i=-\infty}^{-1}\alpha^{i}(|i|^{n}-|i+1|^{n})\\
        &=&\sum_{i=0}^{\infty}\frac{(i+1)^{n}-i^{n}}{\alpha^{i+1}}\leq \sum_{i=1}^{\infty}\frac{i^{n}}{\alpha^{i}}< +\infty.
    \end{eqnarray*}
    
    Then from $(i)$ and $n$-th Jensen formula, we have
    \begin{eqnarray*}\label{3a4}
        T(r, e_{n, \alpha})&=& \sum_{j=1}^{n}N^{(j)}(r, 1_{0}\oslash e_{n, \alpha})\nonumber+O(1)\\
        &=&\frac{1}{2}\sum_{j=1}^{n}\sum_{m=-\lfloor r\rfloor}^{\lfloor r\rfloor}C_{n}^{j}|m|^{n-j}\alpha^{m-1}(\alpha-1)(r-|m|)^{j}+O(1)\nonumber\\
        &=&\frac{1}{2}(\alpha-1)\sum_{m=-\lfloor r\rfloor}^{\lfloor r\rfloor}\alpha^{m-1}\sum_{j=1}^{n}C_{n}^{j}|m|^{n-j}(r-|m|)^{j}+O(1)\nonumber\\
        &=&\frac{1}{2}(\alpha-1)\sum_{m=-\lfloor r\rfloor}^{\lfloor r\rfloor}\alpha^{m-1}(r^{n}-|m|^{n})+O(1)\nonumber\\
        &=&\frac{1}{2}(\alpha-1)\sum_{m=-\lfloor r\rfloor}^{\lfloor r\rfloor}\alpha^{m-1}(r-|m|)\left(\sum_{k=0}^{n-1}r^{k}|m|^{n-k-1}\right)+O(1).
    \end{eqnarray*}
    On one hand,
    \begin{eqnarray*}
        \frac{\alpha^{\lfloor r\rfloor-1}-1}{\alpha-1}\leq\sum_{m=-\lfloor r\rfloor+1}^{\lfloor r\rfloor-1}\alpha^{m-1}\leq \sum_{m=-\lfloor r\rfloor}^{\lfloor r\rfloor}\alpha^{m-1}(r-|m|)\leq \sum_{m=-\lfloor r\rfloor}^{\lfloor r\rfloor}\alpha^{m-1}r\leq \frac{r\alpha^{\lfloor r\rfloor}}{\alpha-1},
    \end{eqnarray*}
    and on the other hand,
    \begin{eqnarray*}
        r^{n-1}\leq \left(\sum_{k=0}^{n-1}r^{k}|m|^{n-k-1}\right)\leq n r^{n-1}.
    \end{eqnarray*}
    Therefore
    \begin{eqnarray*}\label{3a3}
        \frac{r^{n-1}}{2}(\alpha^{\lfloor r\rfloor-1}-1)+O(1)\leq T(r, e_{n, \alpha})\leq \frac{n r^{n}}{2}\alpha^{\lfloor r\rfloor}+O(1),
    \end{eqnarray*}
    from which we can prove the conclusion.
\end{proof}

\section{$n$-th tropical version of logarithmic derivative lemma}
According to the definition, we know that $T(r, f)$ is a continuous and piecewise polynomial function in $[0, \infty)$. However, the properties of non-decreasingness and convexity of the classical tropical characteristic function  may fail in our setting, as we can see from the following example.
\begin{exa}Let
    \begin{eqnarray*}
        f(x)&=& \left\{\begin{array}{ll}
         -(x-m)(x-m-1), & m\leq x\leq m+1,\\
         (x+m)(x+m+1), & -m-1\leq x \leq -m,
         \end{array}\right.
    \end{eqnarray*}
    for all $m\in\mathbb{N}\cup\{0\}$. Then $m(r, f)=-\frac{1}{2}(r-\lfloor r\rfloor)(r-\lfloor r\rfloor-1)$ and $\sum_{j=1}^{2}N^{(j)}(r, f)=N^{(1)}(r, f)=\frac{1}{2}\sum_{i=1}^{\lfloor r\rfloor}2(r-i)$, hence $T(r, f)=-\frac{1}{2}r^{2}+(2\lfloor r\rfloor+\frac{1}{2})r-\lfloor r\rfloor(\lfloor r\rfloor+1)$, which is neither convex nor non-decreasing.
\end{exa}

To imitate the tropical logarithmic derivative lemma for the shift operator, we add some conditions to the general $n$-th tropical meromorphic function $f$.

\begin{definition}\label{def4.2}
Let $f(x)$ be a polynomial with degree $n\in\mathbb{N}$. We say $f(x)$ is \textbf{well defined} if it has following form
\begin{eqnarray*}
    f(x)-f(0)&=&\left\{\begin{array}{ll}
         \sum_{j=1}^{m}a_{2j-1}x^{2j-1}, & \text{if }n=2m-1,\\
         \sum_{j=1}^{m}a_{2j}x^{2j}, & \text{if } n=2m,
         \end{array}\right.
\end{eqnarray*}
where $a_{2j-1}$ and $a_{2j}$ are all non-negative or non-positive numbers for $j=1, \cdots, m$. Especially, any linear function is well defined.
\end{definition}

Example 4.1 shows that the volatility of the polynomials will lead the Characteristic function fluctuating, and, to some degree, this definition is the simplest way to eliminate the volatility of the polynomials and the derivative function of any order is monotonic. Besides, we have the following lemma for the well defined polynomials.

\begin{lemma}\label{4.3.0}
    Let $n\in\mathbb{N}$, $\delta\in\{1, -1\}$ and $f$ be a well defined polynomial of degree $n$, then
    \begin{eqnarray}\label{4dp}
    \frac{d^{l}f(\delta r)}{d r^{l}}\geq 0 \text{ ( or}\leq 0 \text{ )}
\end{eqnarray}
for all $l=1, \cdots, n$ and $r>0$.
\end{lemma}

\begin{proof}
    We follow the notations of Definition \ref{def4.2}. If $n=2m$ for some $m\in\mathbb{N}$, without loss of generality, we assume $a_{2}\neq 0$, and then
    \begin{eqnarray*}
        sgn\left(\frac{d^{l}f(\delta r)}{d r^{l}}\right)&=&sgn\left(\frac{d^{l}(\sum_{j=1}^{m}a_{2j}r^{2j})}{dr^{l}}\right)\\
        &=&sgn\left(a_{2}\right)\geq 0 \text{ ( or}\leq 0 \text{ )}
    \end{eqnarray*}
    for all $l=1, \cdots, n$ and $r>0$.

    If $n=2m-1$ for some $m\in\mathbb{N}$, without loss of generality, we assume $a_{1}\neq 0$, then
    \begin{eqnarray*}
        sgn\left(\frac{d^{l}f(\delta r)}{d r^{l}}\right)&=&sgn\left(\frac{d^{l}(\delta^{n}\sum_{j=1}^{m}a_{2j-1}r^{2j-1})}{dr^{l}}\right)\\
        &=&sgn\left(\delta^{n}a_{1}\right)\geq 0 \text{ ( or}\leq 0 \text{ )}
    \end{eqnarray*}
    for all $l=1, \cdots, n$ and $r>0$.
\end{proof}

We also call a $n$-th tropical meromorphic function $f(x)$ \textbf{ well defined} if $f(x)$ is well defined on each segment. For instance, $e_{n, \alpha}(x)$ is a well defined $n$-th tropical entire function.

Similarly with the result that $T(r, f)$ is a non-decreasing convex function when $f$ is a $1$-st tropical meromorphic function, we have following lemma for the $n$-th setting. However, unlike in the proof when $n=1$ where tropical Cartan identity was used \cite[Theorem 3.8]{Rbook}, our argument is direct and avoids this kind of identity.

\begin{lemma}\label{lemma4.3}
    Let $n\in\mathbb{N}$, then for any well defined $n$-th tropical meromorphic function $f(x)$, we have $T^{(l)}(r^{-}, f)$ is non-negative and non-decreasing for any $l=0, 1, \cdots, n$. In particular, $T(r, f)$ is a non-negative non-decreasing convex function.
\end{lemma}
\begin{proof}

    We first prove $T^{(l)}(r^{-}, f)\geq 0$ for $l=1, \cdots, n$ and a small enough $r>0$ such that $x=0$ is the unique potential singularity of $f$ in $[-2r, 2r]$. Furthermore, because of the continuity of $f$, we can let $f(x)$ be non-negative or non-positive when $x\in(-2r, 0]$ and $f(y)$ be non-negative or non-positive when $y\in[0, 2r)$. In this case
    \begin{eqnarray*}
        T(r, f)=\frac{f^{+}(r)+f^{+}(-r)}{2}+\frac{1}{2}\sum_{j=1, \omega_{f}^{(j)}(0)<0}^{n}|\omega_{f}^{(j)}(0)|r^{j}.
    \end{eqnarray*}
    If $f(x)\leq 0$ in $(-2r, 2r)$ then it is clear that $T^{(l)}(r^{-}, f)\geq 0$. If $f(x)\geq 0$ in $(-2r, 2r)$, we can let $g=-f$, then by applying the $n$-th Jensen formula, we can obtain the conclusion as well. 
    
    If $f(0)=0$, $f(x)\geq 0$ and $f(-x)\leq 0$ when $x\in(0, 2r)$, then there exists $x_{0}\in(0, 2r)$ such that $f^{'}(x_{0})\geq 0$. It follows from Lemma \ref{4.3.0} that $\frac{d^{l}f(r)}{dr^{l}}\geq 0$, thus $T^{(l)}(r^{-}, f)\geq \frac{1}{2}\frac{d^{l}f(r)}{dr^{l}}\geq 0$ holds for all $l=1, \cdots, n$. If $f(0)=0$, $f(x)\leq 0$ and $f(-x)\geq 0$ when $x\in(0, 2r)$, we can also apply the $n$-th Jensen Formula to prove the claim.

    Since $T^{(l)}(r^{-}, f)$ is non-negative for a small $r>0$ and $l=0, \cdots, n$, if $T^{(n)}(r^{-}, f)$ is non-decreasing for all $r>0$, then $T^{(n)}(r^{-}, f)$ is non-negative for all $r>0$, and hence $T^{(n-1)}(r^{-}, f)$ is non-decreasing for all $r>0$, and, by induction, we can get the conclusion. Hence it is sufficient to prove $T^{(n)}(r^{-}, f)$ is non-decreasing for all $r>0$. Since $T^{(n)}(r^{-}, f)$ is clearly piecewise constant function in $(0, \infty)$, thus it is enough to prove $T^{(n)}(r^{+}, f)-T^{(n)}(r^{-}, f)\geq 0$ holds for all $r>0$ now. 
    
    If $f(x)$, $f(-x)\leq 0$ when $x\in(r-\varepsilon, r+\varepsilon)$ for a small quantity $\varepsilon>0$ such that $x=r$ is the potential singularity in $(r-\varepsilon, r+\varepsilon)$, then we have
    \begin{eqnarray*}
        &&T^{(n)}(r^{+}, f)-T^{(n)}(r^{-}, f)=\lim\limits_{r_{1}\rightarrow r^{+}}T^{(n)}(r_{1}, f)-\lim\limits_{r_{2}\rightarrow r^{-}}T^{(n)}(r_{2}, f)\\
        &=&\lim\limits_{r_{1}\rightarrow r^{+}}\left( \sum_{j=1}^{n}N_{(-r, r)}^{(j)}(r_{1}, f)+\frac{1}{2}\sum_{j=1,\omega_{f}^{(j)}(\pm r)<0 }^{n}|\omega_{f}^{(j)}(\pm r)|(r_{1}-r)^{j}\right)^{(n)}\\
        &&-\lim\limits_{r_{2}\rightarrow r^{-}}\left(\sum_{j=1}^{n}N^{(j)}(r_{2}, f)\right)^{(n)}\\
        &=& \frac{1}{2}\lim\limits_{r_{1}\rightarrow r^{+}}\left(\sum_{j=1,\omega_{f}^{(j)}(\pm r)<0 }^{n}|\omega_{f}^{(j)}(\pm r)|(r_{1}-r)^{j}\right)^{(n)}\\
        \end{eqnarray*}
        \begin{eqnarray*}
        &=&\frac{1}{2}\sum_{\omega_{f}^{(n)}(\pm r)<0}n!|\omega_{f}^{(n)}(\pm r)|\geq 0.
    \end{eqnarray*}
    
    If $f(x)$, $f(-x)\geq 0$ when $x\in(r-\varepsilon, r+\varepsilon)$ for a small quantity $\varepsilon>0$, then from $n$-th Jensen formula, we have
    \begin{eqnarray*}
        T^{(n)}(r^{+}, f)-T^{(n)}(r^{-}, f)&=&T^{(n)}(r^{+}, -f)-T^{(n)}(r^{-}, -f)\\
        &=&\frac{1}{2}\sum_{\omega_{f}^{(n)}(\pm r)>0}n!|\omega_{f}^{(n)}(\pm r)|\geq 0.
    \end{eqnarray*} 

     If $f(x)\geq 0$ and $f(-x)\leq 0$ when $x\in(r-\varepsilon, r+\varepsilon)$ for a small quantity $\varepsilon>0$, then 
    \begin{eqnarray*}
        &&T^{(n)}(r^{+}, f)-T^{(n)}(r^{-}, f)\\
        &=&\lim\limits_{r_{1}\rightarrow r^{+}}\left( \frac{1}{2}f(r_{1})+\sum_{j=1}^{n}N_{(-r, r)}^{(j)}(r_{1}, f)+\sum_{j=1,\omega_{f}^{(j)}(\pm r)<0 }^{n}|\omega_{f}^{(j)}(\pm r)|(r_{1}-r)^{j}\right)^{(n)}\\
        &&-\lim\limits_{r_{2}\rightarrow r^{-}}\left(\frac{1}{2}f(r_{2})+\sum_{j=1}^{n}N^{(j)}(r_{2}, f)\right)^{(n)}\\
        &=&\frac{1}{2}\sum_{\omega_{f}^{(n)}(\pm r)<0}n!|\omega_{f}^{(n)}(\pm r)|\geq 0.
    \end{eqnarray*}

    The case where $f(-x)\geq 0$ and $f(x)\leq 0$, when $x\in(r-\varepsilon, r+\varepsilon)$ for a small quantity $\varepsilon>0$, can be proved similarly.

    It remains to prove the case that $f(\delta r)=0$ and there is a small positive quantity $\varepsilon$ such that $f(\delta r+\varepsilon)f(\delta r-\varepsilon)<0$ for $\delta=1$ or $-1$. If $f(\pm r)=0$, $f(x)<0$ when $x\in (-r-\varepsilon, -r)\cup (r, r+\varepsilon)$ and $f(y)>0$ when $y\in (-r, -r+\varepsilon)\cup (r-\varepsilon, r)$, then $f^{'}(r^{\pm})<0$ and $f^{'}(-r^{\pm})=-\frac{df(-r^{\pm})}{dr}>0$, which leads to $f^{(n)}(r^{\pm})<0$ and $\frac{d^{n}f(-r^{\pm})}{dr^{n}}=(-1)^{n}f^{(n)}(-r^{\pm})<0$ from Lemma \ref{4.3.0}. Thus
    \begin{eqnarray*}
        T^{(n)}(r^{+}, f)-T^{(n)}(r^{-}, f)
        &\geq&-\frac{1}{2}\lim\limits_{r_{2}\rightarrow r^{-}}\left(f(r_{2})+f(-r_{2})\right)^{(n)}\\
        &=& -\frac{1}{2}\left(f^{(n)}(r^{-})+(-1)^{n}f^{(n)}(-r^{+})\right)\\
        &\geq& 0.
    \end{eqnarray*}

    The other cases can be proved with the same manner, we omit the details here.
\end{proof}

As we know, the logarithmic derivative lemma is crucial to prove the second main theorem in the classical meromorphic function field, hence a tropical analogue with the shift operator was constructed \cite{Tnt, Rbook, Pro}. However, we found that the conditions in this tropical analogue can be relaxed to obtain a pointwise estimate which will help us to prove the relationship between $\sum_{i=0}^{n}N(r, 1_{0}\oslash f_{i})$ and the ramification term $N(r, C_{0}(f_{0}, \cdots, f_{m}))$ (Theorem \ref{theorem6.2} and Corollary \ref{cor6.3}). In addition, the proof is much more direct than previously obtained.
\begin{lemma}\label{lemma4.4}
    Let $n\in\mathbb{N}$ and $f$ be a well defined $n$-th tropical meromorphic function. Then for any $\alpha>1$ and $c\neq 0$, we have
    \begin{eqnarray}\label{4lp}
        \left|f(\delta r+c)\oslash f(\delta r)\right|\leq \frac{32|c|}{(\alpha-1)(r+|c|)}\left(T(\alpha(r+|c|), f)+\frac{|f(0)|}{2}\right)
    \end{eqnarray}
    holds for all $r>\max\{2|c|, \frac{3-\alpha}{\alpha-1}|c|\}$ where $\delta=\pm 1$.
\end{lemma}
\begin{proof}
    Let $\rho=\frac{1}{2}(\alpha+1)(r+|c|)<\alpha(r+|c|)$, where, without loss of generality, we assume $c>0$. Define $\{x_{i}\}_{i=1}^{\nu}$ and $\{x_{i}\}_{i=1}^{\nu_{1}}$ that consist of all singularities of $f$ in $(r, r+c)$ and $(r, \rho)$ separately. Since $r>\frac{3-\alpha}{\alpha-1}c$, then $\frac{c}{\rho-r-c}\leq 1$, hence $0<\frac{c}{\rho-r}\leq\frac{c}{\rho-x_{i}}\leq\frac{c}{\rho-r-c}\leq 1$ for $i=1, \cdots, \nu$, and due to $r>2c$, we have $\tau_{f}^{(j)}(x_{i})=\omega_{f}^{(j)}(x_{i})$ for all $j=1, \cdots, n$ and $i=1, \cdots, \nu_{1}$. Then by applying Lemma \ref{lemma3.1} twice, we have
    \begin{eqnarray*}
        &&|f(r+c)-f(r)|=\left|\sum_{j=1}^{n}\frac{f^{(j)}(r^{+})}{j!}(r+c-r)^{j}-\sum_{j=1}^{n}\sum_{i=1}^{\nu}\omega_{f}^{(j)}(x_{i})(r+c-x_{i})^{j}\right|\\
        &\leq&\left|\sum_{j=1}^{n}\frac{f^{(j)}(r^{+})}{j!}(\rho- r)^{j}\left(\frac{c}{\rho- r}\right)^{j}\right|+\left|\sum_{j=1}^{n}\sum_{i=1}^{\nu}\omega_{f}^{(j)}(x_{i})(\rho-x_{i})^{j}\left(\frac{r+c-x_{i}}{\rho-x_{i}}\right)^{j}\right|\\
        &\leq&\frac{c}{\rho-r-c}\left(\left|\sum_{j=1}^{n}\frac{f^{(j)}(r^{+})}{j!}(\rho-r)^{j}\right|+\sum_{j=1}^{n}\sum_{i=1}^{\nu}|\omega_{f}^{(j)}(x_{i})|(\rho-x_{i})^{j}\right)\\
        &\leq&\frac{c}{\rho-r-c}\left(\left|f(\rho)-f(r)-\sum_{j=1}^{n}\sum_{i=1}^{\nu_{1}}\omega_{f}^{(j)}(x_{i})(\rho-x_{i})^{j}\right|\right)\\
        &&+\frac{2c}{\rho-r-c}\left(\sum_{j=1}^{n}\left(N_{(r, r+c)}^{(j)}(\rho, f)+N_{(r, r+c)}^{(j)}(\rho, -f)\right)\right)\\
        &\leq&\frac{c}{\rho-r-c}\left(|f(\rho)|+|f(x)|+2\sum_{j=1}^{n}\left(N_{(r, \rho)}^{(j)}(\rho, f)+N_{(r, \rho)}^{(j)}(\rho, -f)\right)\right)\\
        &&+\frac{2c}{\rho-r-c}\left(\sum_{j=1}^{n}\left(N_{(r, r+c)}^{(j)}(\rho, f)+N_{(r, r+c)}^{(j)}(\rho, -f)\right)\right)\\
        &\leq&\frac{2c}{\rho-r-c}(m(\rho, f)+m(\rho, -f)+m(r, f)+m(r, -f))\\
        &&+\frac{4c}{\rho-r-c}\sum_{j=1}^{n}\left(N^{(j)}(\rho, f)+N^{(j)}(\rho, -f)\right)\\
        &\leq&\frac{16c}{\rho-r-c}\left(T(\rho, f)+\frac{|f(0)|}{2}\right)\\
        &\leq&\frac{32c}{(\alpha-1)(r+c)}\left(T(\alpha(r+c), f)+\frac{|f(0)|}{2}\right).
    \end{eqnarray*}
    The second inequality holds because $f(x)$ is well defined, which implies $f^{(j)}(r^{+})$ is non-negative or non-positive for all $j=1, \cdots, n$. We can also do a similar analysis for $|f(-r+c)-f(-r)|$ and get the same inequality. Hence the assertion follows.
\end{proof}

\noindent\textbf{Remark}: Example 4.1 shows $T(r, f)$ may not be convex over $r\in(0, \infty)$ in the $n$-th setting, but several examples suggest us that $T(r, f)$ may be non-increasing outside of a set of finite measure, but we can not prove this yet. Another problem to discard the assumption that $f$ should be well defined is to prove Lemma \ref{lemma4.4}. It is worth noting that Lemma \ref{lemma3.1} is easier applied than the $n$-th Poisson-Jensen formula when we prove Lemma \ref{lemma4.4}.

Growth assumption is necessary to study the characteristic function with a shift operator, and the following lemma is frequently applied to prove various difference analogues of the logarithmic derivative lemma. Some generalizations of the Lemma can be seen in \cite{Thyp, Log}.

\begin{lemma}\label{lemma4.5}
    (\cite[Lemma 8.3]{Hyp}) Let $T:[0, +\infty)\rightarrow [0, +\infty)$ be a non-decreasing continuous function and let $c\in(0, \infty)$. If the hyper-order of $T$ is strictly less than one. i.e.,
    \begin{eqnarray}\label{4sup}
        \limsup\limits_{r\rightarrow \infty}\frac{\log\log T(r)}{\log r}:=\rho_{2}<1
    \end{eqnarray}
    and $\sigma\in(0, 1-\rho_{2})$, then
    \begin{eqnarray*}
        T(r+c)=T(r)+o\left(\frac{T(r)}{r^{\sigma}}\right)
    \end{eqnarray*}
    where $r$ runs to infinity outside of a set of finite logarithmic measure.
\end{lemma}

By combining Lemma \ref{lemma4.4} and \ref{lemma4.5}, we can get following estimation. A similar proof as in Proposition 4.2 \cite{Secm} applies. For the convenience of the reader, however, we recall the proof briefly.

\begin{theorem}[\textbf{Pointwise $n$-th tropical logarithmic derivative estimate}]\label{theorem 4.6}
     Let $n\in\mathbb{N}$ and $c$ be a nonzero real number. If $f$ is a $n$-th well defined tropical meromorphic function on $\mathbb{R}$ with
    \begin{eqnarray}\label{4as}
        \rho_{2}(f):=\rho_{2}<1,
    \end{eqnarray}
    and $\sigma\in(0, 1-\rho_{2})$, then 
    \begin{eqnarray*}
        \left|f(\delta r+c)\oslash f(\delta r)\right|=o\left(\frac{T(r, f)}{r^{\sigma}}\right)
    \end{eqnarray*}
    where $\delta=\pm1$ and $r$ runs to infinity outside of a set of finite logarithmic measure.
\end{theorem}
\begin{proof}
    For any $\varepsilon>0$, let
    \begin{eqnarray*}
        \alpha:=1+\frac{1}{(\log T(r+|c|, f))^{1+\varepsilon}}.
    \end{eqnarray*}
    Then by applying the generalized Borel Lemma as given in Lemma 3.3.1 \cite{Ye} and Lemma \ref{lemma4.5}, we have $T(\alpha(r+|c|))\leq CT(r+|c|, f)=CT(r, f)+o(T(r, f)/r^{\sigma})$ for a positive constant $C$ and all $r$ outside of a set of finite logarithmic measure. 
    
    We may now fix $\varepsilon>0$ such that $\sigma=1-(\rho_{2}+\varepsilon)(1+\varepsilon)$. Then under the assumption (\ref{4as}), we get
    \begin{eqnarray*}
        \frac{1}{(\alpha-1)(r+|c|)}=\frac{(\log T(r+|c|, f))^{1+\varepsilon}}{r+|c|}=o\left(\frac{1}{(r+|c|)^{\sigma}}\right), \quad (r\rightarrow +\infty),
    \end{eqnarray*}
    and hence
    \begin{eqnarray*}
        \frac{(3-\alpha)|c|}{r(\alpha-1)}=\frac{1}{(\alpha-1)(r+|c|)}\frac{(r+|c|)(3-\alpha)|c|}{r}=o\left(\frac{1}{(r+|c|)^{\sigma}}\right), \quad (r\rightarrow +\infty)
    \end{eqnarray*}
    which implies $r\geq \max\left\{ 2|c|, \frac{3-\alpha}{\alpha-1}|c|\right\}$ when $r$ is large enough. Therefore, Lemma~\ref{lemma4.4} yields
    \begin{eqnarray*}
        |f(\delta r+c)\oslash f(\delta r)|&\leq& 32|c|o\left(\frac{1}{(r+|c|)^{\sigma}}\right)\left(CT(r+|c|, f)+\frac{|f(0)|}{2}\right)\\
        &=&o\left(\frac{T(r, f)}{r^{\sigma}}\right)
    \end{eqnarray*}
    as $r$ runs to infinity outside of a set of finite logarithmic measure.
\end{proof}
Then the following $n$-th tropical version of logarithmic derivative lemma is a corollary of the Theorem \ref{theorem 4.6}.

\begin{corollary}[\textbf{$n$-th tropical version of logarithmic derivative lemma}]\label{cor4.7}
     Let $n\in\mathbb{N}$ and $c$ be a nonzero real number. If $f$ is a well defined $n$-th tropical meromorphic function on $\mathbb{R}$ with
    \begin{eqnarray*}
        \rho_{2}(f):=\rho_{2}<1,
    \end{eqnarray*}
    and $\sigma\in(0, 1-\rho_{2})$, then 
    \begin{eqnarray*}
        m(r, f(x+c)\oslash f(x))=o\left(\frac{T(r, f)}{r^{\sigma}}\right)
    \end{eqnarray*}
    where $r$ runs to infinity outside of a set of finite logarithmic measure.
\end{corollary}
\begin{proof}
    Since
    \begin{eqnarray*}
        m\left(r, f(x+c)\oslash f(x)\right)&=&\frac{(f(r+c)-f(r))^{+}+(f(-r+c)-f(-r))^{+}}{2}\\
        &\leq& \frac{|f(r+c)-f(r)|+|f(-r+c)-f(-r)|}{2},
    \end{eqnarray*}
    then we can immediately get the conclusion by applying Theorem \ref{theorem 4.6}.
\end{proof}

\section{$n$-th tropical holomorphic curve}
Now we extend some notations of the $n$-th tropical setting to the tropical projective space. Simply speaking, the tropical space $\mathbb{T}\mathbb{P}^{m}$ ($m\in\mathbb{N}$) is given by a equivalence relation $\sim$ in $\mathbb{R}_{\max}^{m+1}\slash \{0_{0}\}$ which is denoted by $(a_{0}, \cdots, a_{m})\sim(b_{0}, \cdots, b_{m})$ if and only if $(a_{0}, \cdots, a_{m})=\lambda\otimes(b_{0}, \cdots, b_{m})$ for some $\lambda\in\mathbb{R}$, and $[a_{0}: \cdots: a_{m}]$ is denoted as the equivalence class of $(a_{0}, \cdots, a_{m})$. Letting $n\in\mathbb{N}$, we denote
\begin{equation*}
        f=[f_{0}: \cdots: f_{m}]:\mathbb{R}\rightarrow \mathbb{T}\mathbb{P}^{m}
    \end{equation*}
    as a \textbf{$n$-th tropical holomorphic map (or curve)} where $f_{0}, \cdots, f_{m}$ are $n_{i}$-th tropical entire functions with $\max_{0\leq i\leq m}\{n_{i}\}=n$ and $\{n_{i}\}\subset \mathbb{N}\cup\{0\}$. Furthermore, we say $\textbf{f}=(f_{0}, \cdots, f_{m})\in [f_{0}: \cdots :f_{m}]$ is a reduced representation of $f$ if $f_{0}, \cdots, f_{m}$ do not have any $j$-th roots which are common to all of them.

 We also follow the definition of the classical tropical Cartan characteristic function of $f$.
    \begin{definition}
        \cite[Definition 4.2]{Pro} Let $n, m\in\mathbb{N}$ and $f:\mathbb{R}\rightarrow \mathbb{T}\mathbb{P}^{m}$ be a $n$-th tropical holomorphic curve with a reduced representation $\textbf{f}=(f_{0}, \cdots, f_{m})$, then
        \begin{eqnarray*}
            T_{\textbf{f}}(r)=\frac{1}{2}(F(r)+F(-r))-F(0),\quad F(x)=\max\limits_{0\leq i\leq m}\{f_{i}(x)\},
        \end{eqnarray*}
        is said to be the tropical Cartan characteristic function of $f$.
    \end{definition}

Similarly with the Proposition 4.3, 4.4 and Lemma 4.7 in \cite{Pro}, we have following three propositions.
\begin{proposition}\label{p5.2}
    The tropical Cartan charateristic function $T_{\textbf{f}}(r)$ is independent of the reduced representation of the $n$-th tropical holomorphic curve $f$.
\end{proposition}
\begin{proof}
    For any $n$-th tropical holomorphic curve $f:\mathbb{R}\rightarrow\mathbb{T}\mathbb{P}^{m}$, let $\textbf{f}=(f_{0}, \cdots, f_{m})$ and $\textbf{g}=(g_{0}, \cdots, g_{m})$ be two arbitrary reduced representations of $f$, then we can find a tropical function $\lambda(x)$ such that $f_{i}(x)=g_{i}(x)+\lambda(x)$ for any $x\in\mathbb{R}$ and $i=0, \cdots, m$. Since $\omega_{f_{i}}^{(j)}(x)\geq 0$ and $\omega_{g_{i}}^{(j)}(x)\geq 0$ for all $x\in\mathbb{R}$, $i=0,\cdots,m$ and $j=1, \cdots, n$, if $\lambda(x)$ has a $j_{0}$-th root or pole at $x_{0}\in\mathbb{R}$ for some $j_{0}\in\{1, \cdots, n\}$, that is,  $\omega_{\lambda}^{(j_{0})}(x_{0})>0$ or $<0$, then
    \begin{eqnarray*}
        \omega_{f_{i}}^{(j_{0})}(x_{0})=\omega_{g_{i}}^{(j_{0})}(x_{0})+\omega_{\lambda}^{(j_{0})}(x_{0})>0
    \end{eqnarray*}
    or
    \begin{eqnarray*}
        \omega_{g_{i}}^{(j_{0})}(x_{0})=\omega_{f_{i}}^{(j_{0})}(x_{0})-\omega_{\lambda}^{(j_{0})}(x_{0})>0
    \end{eqnarray*}
    for all $i=0, \cdots, m$. But this means that $z_{0}$ is a $j_{0}$-th root for all $f_{i}$ or $g_{i}$ which is a contradiction. Hence, $\lambda(x)$ is a tropical function without any singularities. Then by Proposition \ref{p2.2}, we have $(\lambda(r)+\lambda(-r))/2=\lambda(0)$, and then $T_{\textbf{f}}(r)=T_{\textbf{g}}(r)$.
\end{proof}
Thus we can adopt the notation $T_{f}(r)$ instead of $T_{\textbf{f}}(r)$ just like in the classical setting.
\begin{proposition}\label{p5.3}
    Let $n\in\mathbb{N}$ and $g=g_{1}\oslash g_{0}$ be a $n$-th tropical meromorphic function, where $g_{0}$ and $g_{1}$ are tropical entire and do not have any common $j$-th roots, and define $f=[g_{0}: g_{1}]$. Then
    \begin{eqnarray*}
        T_{f}(r)=T(r, g)-g^{+}(0).
    \end{eqnarray*}
\end{proposition}
\begin{proof}
Since $F(x)=\max\{g_{0}(x), g_{1}(x)\}=(g_{1}- g_{0})^{+}(x)+g_{0}(x)$, it follows by applying the $n$-th Jensen formula that
\begin{eqnarray*}
    T_{f}(r)&=&\frac{1}{2}\sum_{\delta=\pm 1}\left((g_{1}- g_{0})^{+}(\delta r)+g_{0}(\delta r)\right)-(g_{1}- g_{0})^{+}(0)-g_{0}(0)\\
    &=&m(r, g_{1}- g_{0})+m(r, g_{0})-m(r, - g_{0})-(g_{1}- g_{0})^{+}(0)-g_{0}(0)\\
    &=&m(r, g_{1}- g_{0})+\sum_{j=1}^{n}\left(N^{(j)}(r, - g_{0})-N^{(j)}(r, g_{0})\right)-(g_{1}- g_{0})^{+}(0)\\
    &=&m(r, g_{1}- g_{0})+\sum_{j=1}^{n}N^{(j)}(r, g_{1}- g_{0})-(g_{1}- g_{0})^{+}(0)\\
    &=&T(r, g_{1}- g_{0})-(g_{1}- g_{0})^{+}(0).
\end{eqnarray*}
\end{proof}

\begin{proposition}\label{p5.4}
Let $n, m\in\mathbb{N}$ and $f:\mathbb{R}\rightarrow \mathbb{T}\mathbb{P}^{m}$ be a non-constant $n$-th tropical holomorphic curve with a reduced representation $\textbf{f}=(f_{0}, \cdots, f_{m})$. Then
    \begin{eqnarray}\label{div}
        T\left(r, \frac{f_{i}}{f_{l}}\oslash\right)\leq T_{f}(r)+O(1),
    \end{eqnarray}
    for any $i, l\in\{0, \cdots, m\}$.
\end{proposition}
\begin{proof} 
   Assume $u(x)$ is a $n_{1}$-th ($n_{1}\leq n$) tropical entire function that satisfies 
   \begin{eqnarray}\label{5a111}
       \omega_{u}^{(j)}(x)=\min\{\omega_{f_{i}}^{(j)}(x), \omega_{f_{l}}^{(j)}(x)\}
   \end{eqnarray}
   for all $x\in\mathbb{R}$ and $j=1, \cdots, n$, and $u(0)=f_{l}(0)$. 
   Letting $h_{i}(x):=f_{i}(x)-u(x)$ and $h_{l}(x)=f_{l}(x)-u(x)$, it follows that $h_{i}(x)$ and $h_{l}(x)$ do not have any common $j$-th roots. Otherwise, say $\omega_{h_{i}}^{(j)}(x_{0})>0$ and $\omega_{h_{l}}^{(j)}(x_{0})>0$ hold simultaneously for some $j=1, \cdots, n$ and $x_{0}\in\mathbb{R}$. Then $\omega_{u}^{(j)}(x_{0})=\omega_{f_{i}}^{(j)}(x_{0})-\omega_{h_{i}}^{(j)}(x_{0})=\omega_{f_{l}}^{(j)}(x_{0})-\omega_{h_{l}}^{(j)}(x_{0})<\min\{\omega_{f_{i}}^{(j)}(x_{0}), \omega_{f_{l}}^{(j)}(x_{0})\}$ which contradicts (\ref{5a111}), hence, from Proposition \ref{p5.3}, we have
    \begin{eqnarray*}
        T\left(r, \frac{f_{i}}{f_{l}}\oslash\right)&=&T\left(r, \frac{h_{i}}{h_{l}}\oslash\right)\\
        &=&\frac{1}{2}\sum_{\delta=\pm 1}\max\{h_{i}(\delta r), h_{l}(\delta r)\}+O(1)\\
        &=&\frac{1}{2}\sum_{\delta=\pm 1}\left(\max\{f_{i}(\delta r), f_{l}(\delta r)\}-u(\delta r)\right)+O(1)\\
        &\leq&\frac{1}{2}\sum_{\delta=\pm 1}\left(\max\limits_{0\leq k\leq m}\{f_{k}(\delta r)\}-u(\delta r)\right)+O(1)\\
        &\leq&T_{f}(r)-\frac{1}{2}\sum_{\delta=\pm 1}u(\delta r)+O(1)\\
        &\leq&T_{f}(r)+O(1).
    \end{eqnarray*}
    The last inequality follows from Proposition \ref{p2.2} with the fact that $u(x)$ is a tropical entire function. Hence the proof is complete.
\end{proof}

A homogeneous tropical polynomial of degree $d$ in $\mathbb{T}\mathbb{P}^{m}$ is defined by
\begin{eqnarray}
    P(x)=\bigoplus_{I_{i}\in \mathcal{J}_{d}}\alpha_{I_{i}}\otimes x^{I_{i}}=\bigoplus_{\sum_{k=0}^{m}i_{k}=d}\alpha_{i_{0}, \cdots, i_{m}}\otimes\bigotimes_{k=0}^{m} x_{i}^{\otimes i_{k}}
\end{eqnarray}
where $x=(x_{0}, \cdots, x_{m})$, $\mathcal{J}_{d}$ is the set of all $m+1$ tuples $I_{i}=(i_{0}, \cdots, i_{m})\in\mathbb{N}_{0}^{n+1}$ with $\sum_{k=0}^{m}i_{k}=d$, and $\alpha_{I_{i}}\in \mathbb{R}\cup\{-\infty\}$ such that not all are equal to $0_{0}$ \cite{Thyp}. Now we can let $P(x)$ act on the holomorphic curve $f\in$ $\mathbb{T}\mathbb{P}^{m}$, and get a composition function
\begin{eqnarray*}
P\circ f=\bigoplus_{\sum_{k=0}^{m}i_{k}=d}\alpha_{i_{0}, \cdots, i_{m}}\otimes\bigotimes_{k=0}^{m} f_{i}^{\otimes i_{k}}.
\end{eqnarray*}

In the classical setting, $P\circ f$ is always a tropical entire function, that is $N(r, P\circ f)\equiv 0$, which implies that $T_{f}(r)$ can be dominated only by $N(r, 1_{0}\oslash P\circ f)$ in the tropical second main theorem, such as \cite[Theorem 5.2]{Secm}, \cite[Theorem 4.8]{Thyp} and \cite[Corollary 3.5]{Tnp}. However, this may not hold in the general $n$-th setting and it is even possible to get $T_{f}(r)=o(\sum_{j=1}^{n}N^{(j)}(r, P\circ f))$, as we can see from the following example.

\begin{exa}
    Let $P(x)=x_{0}\oplus x_{1}$ and $f=[f_{0}: f_{1}]$ be a $2$-nd holomorphic curve where $f_{0}$ and $f_{1}$ are defined as $f_{0}(x)=sgn(x)x^{2}$ and
    \begin{eqnarray*}
         f_{1}(x)&=&\left\{\begin{array}{ll}
         -\frac{1}{2}, & x\leq 0,\\
         (4k+2)x-2k(2k+2)-\frac{1}{2}, & x\in[2k, 2k+2),
         \end{array}\right.\\
         &=&\left\{\begin{array}{ll}
         -\frac{1}{2}, & x\leq 0,\\
         (4\left\lfloor \frac{x}{2}\right\rfloor+2)x-2\left\lfloor \frac{x}{2}\right\rfloor(2\left\lfloor \frac{x}{2}\right\rfloor+2)-\frac{1}{2}, & x>0,
         \end{array}\right.
    \end{eqnarray*}
    where $k\in\mathbb{N}\cup\{0\}$. It is clear that $T_{f}(r)=\frac{1}{2}r^{2}(1+o(1))$. Besides, when $x>0$, $f_{0}(x)=f_{1}(x)$ gives $x=2k+1\pm\frac{\sqrt{2}}{2}\in[2k, 2k+2)$, and then we can get $x=2k+1-\frac{\sqrt{2}}{2}$, $k\in\mathbb{N}\cup\{0\}$ are all $P\circ f$'s $2$-nd poles with multiplicities $1$. Hence
    \begin{eqnarray*}
        \sum_{j=1}^{2}N^{(j)}(r, P\circ f)&\geq&\frac{1}{2}\sum_{i=1}^{\lfloor\frac{r}{2}\rfloor}(r-2i)^{2}\geq 2\sum_{i=1}^{\lfloor\frac{r}{2}\rfloor}\left(\left\lfloor\frac{r}{2}\right\rfloor-i\right)^{2}=\frac{r^{3}}{12}(1+o(1)),
    \end{eqnarray*}
    from which we can see that $T_{f}(r)=o\left(\sum_{j=1}^{2}N^{(j)}(r, P\circ f)\right)$.
\end{exa}

In the classical tropical setting, the first author and Tohge applied the notion of the linear independence in Gondran-Minoux sense due to Gondran and Minoux \cite{Gon1, Gon2} for some tropical linear functions, and then defined the completeness of the classical tropical linear combinations \cite{Pro}, and, with the definition's help, a new notation of the degree of degeneracy $\lambda=ddg(Q)$ was given for some tropical linear combinations $Q$. Lastly, a tropical analogues of Nochka's extension of Cartan's second main theorem was proved. In this theorem, the functions $g_{0}, \cdots, g_{m}$ should be linearly independent, and the conclusion includes $\lambda$ over the semi-ring $\mathbb{R}_{\max}=\mathbb{R}\cup \{-\infty\}$ \cite[Theorem 6.6]{Pro}. 

Subsequently, In \cite{Thyp}, Cao and Zheng applied the concepts of algebraic (respectively linear) independence in the Gondran-Minoux sense and tropical algebraic (respectively linear) nondegeneracy for some tropical holomorphic curves, and then proved the second main theorem with tropical hypersurfaces over $\mathbb{R}_{\max}$. There also is a condition that the holomorphic curve should be tropically algebraically nondegenerate, and the conclusion also includes $\lambda$.

Recently, over $\mathbb{R}_{\max}$, Halonen, the first author and Filipuk introduced two definitions 
$\psi(P, f)$ and $\Psi(P, f)$ for a homogeneous tropical polynomial $P$ and a tropical holomorphic curve $f$ \cite{Tnp}, making the estimate in Theorem 4.8 \cite{Thyp} tighter. In the theorem, the growth assumption can be discarded \cite[Corollary 3.5]{Tnp}. However, their next result \cite[Corollary 3.6]{Tnp} includes the growth assumption, which is an analogue of the one in \cite[Theorem 4.8]{Thyp}. 

Actually, if we assume $f:\mathbb{R}\rightarrow \mathbb{T}\mathbb{R}^{m}$ is a holomorphic curve with a reduced representation $\textbf{f}=(f_{0}, \cdots, f_{m})$ and $P(x)$ is a tropical homogeneous polynomial, then the previously known tropical second main theorems in tropical projective space \cite[Theorem 6.2 and 6.3]{Pro} and \cite[Theorem 4.8]{Thyp} can be divided into two independent parts. The first part is the relationship of $T_{f}(r)$ and $N(r, 1_{0}\oslash P\circ f)$ which will be studied in the remainder of this section and we consider it as the tropical second main theorem. The second is the relationship of $\sum_{i=0}^{m}N(r, 1_{0}\oslash f_{i})$ and $N(r, 1_{0}\oslash C_{0}(f_{0}, \cdots, f_{m}))$ where $C_{0}(f_{0}, \cdots, f_{m})$ is the tropical Casoratian of $f_{0}, \cdots, f_{m}$ which be illustrated in Section 6. Both these two relationships will be explored as the special cases in the $n$-th tropical setting. One of the benefits that we divided it as two parts is that we can simplify the description of the tropical second theorem, meaning that we can discard the assumption that $g_{0}, \cdots, g_{m}$ should be linearly independent in \cite[Theorem~6.6]{Pro}, and the assumption that the holomorphic curve should be tropical algebraically nondegerated in \cite[Theorem~4.8]{Thyp}.

From example 5.5, we know that the term $\sum_{j=1}^{n}N^{(j)}(r, P\circ f)$ can not be ignored in the $n$-th tropical setting, and hence we have the tropical second main theorem below.

\begin{theorem}[\textbf{$n$-th Second main theorem with a tropical homogeneous polynomial}]\label{theorem5.6}
     Let $n, m\in\mathbb{N}$ and $f:\mathbb{R}\rightarrow \mathbb{T}\mathbb{P}^{m}$ be a non-constant $n$-th tropical holomorphic curve with a reduced representation $\textbf{f}=(f_{0}, \cdots, f_{m})$, and let $P(x)$ be a tropical homogeneous polynomial of degree $d>0$ of the form
    \begin{eqnarray*}
        P(x)=\left(\bigoplus_{k=0}^{m}\alpha_{k}\otimes x_{k}^{\otimes d}\right)\oplus\left(\bigoplus_{\sum_{k=0}^{m}i_{k}=d, i_{k}<d}\alpha_{i_{0}, \cdots, i_{m}}\otimes \bigotimes_{k=0}^{m}x_{i}^{\otimes i_{k}}\right)
    \end{eqnarray*}
    where coefficients $\alpha_{k}\in\mathbb{R}$, $k=0, \cdots, m$ and $\alpha_{i_{0}, \cdots, i_{m}}\in\mathbb{R}_{\max}$, $\sum_{k=0}^{m}i_{k}=d, i_{k}<d$. Then
    \begin{eqnarray}
        T_{f}(r)=\frac{1}{d}\left(\sum_{j=1}^{n}N^{(j)}\left(r, \frac{1_{0}}{P\circ f}\oslash\right)-\sum_{j=1}^{n}N^{(j)}\left(r, P\circ f\right)\right)+O(1).
    \end{eqnarray}
\end{theorem}
\begin{proof}
    According to the $n$-th Jensen formula, we obtain
    \begin{eqnarray*}
        &&\sum_{j=1}^{n}N^{(j)}\left(r, \frac{1_{0}}{P\circ f}\oslash\right)-\sum_{j=1}^{n}N^{(j)}(r, P\circ f)\\
        &=&m(r, P\circ f)-m\left(r, \frac{1_{0}}{P\circ f}\oslash\right)+O(1)\\
        &=&\frac{1}{2}\sum_{\delta=\pm 1}P\circ f(\delta r)+O(1).
    \end{eqnarray*}
Let $\beta=\max\left\{\max\limits_{0\leq k\leq m}\{\alpha_{k}\}, \max\limits_{\sum_{k=0}^{m}i_{k}=d}\{\alpha_{i_{0}, \cdots, i_{m}}\}\right\}$ and $\gamma=\min\limits_{0\leq k\leq m}\{\alpha_{k}\}$. Because $\alpha_{k}\in\mathbb{R}$, $k=0, \cdots, m$, it follows that $\beta, \gamma>0_{0}$ and $\gamma+d\max\limits_{0\leq i\leq m}f_{i}(x)\leq P\circ f(x)\leq \beta+d\max\limits_{0\leq i\leq m}f_{i}(x)$ for all $x\in\mathbb{R}$. This gives
    \begin{eqnarray*}
        \frac{1}{2}\sum_{\delta=\pm 1}P\circ f(\delta r)&=&\frac{1}{2}\sum_{\delta=\pm 1}d\max\limits_{0\leq i\leq m}f_{i}(\delta r)+O(1)\\
        &=&dT_{f}(r)+O(1).
    \end{eqnarray*}
    Thus the assertion follows by combining these two equations.
\end{proof}

We denote the tropical homogeneous Fermat polynomial of degree $d$ with constant coefficients as
\begin{eqnarray*}
P(x)=\bigoplus\limits_{i=0}^{m}\alpha_{i}\otimes x_{i}^{\otimes d},
\end{eqnarray*}
where $\alpha_{i}\in\mathbb{R}$. If $f$ is a $1$-st tropical holomorphic curve, then it is clear that $N(r, P\circ f)\equiv 0$, and then we have the following Corollary.

\begin{corollary}
Let $m\in\mathbb{N}$ and $f:\mathbb{R}\rightarrow \mathbb{T}\mathbb{P}^{m}$ be a $1$-st tropical holomorphic curve with a reduced representation $\textbf{f}=(f_{0}, \cdots, f_{m})$, and $P(x)=\bigoplus\limits_{i=0}^{m}\alpha_{i}\otimes x_{i}^{\otimes d}$ where $d>0$ and $\{\alpha_{i}\}_{i=0}^{m}\subset \mathbb{R}$. Then
    \begin{eqnarray}
        T_{f}(r)=\frac{1}{d}N\left(r, \frac{1_{0}}{P\circ f}\oslash\right)+O(1).
    \end{eqnarray}
    \end{corollary}

Since we have already extended piecewise linear functions to piecewise polynomial functions, we can also extend tropical homogeneous polynomials to ordinary homogeneous polynomials.

Consider the homogeneous Fermat polynomial of degree $n$ with constant coefficients as
\begin{eqnarray*}
    \mathcal{P}(x)=\sum_{i=0}^{m}\alpha_{i}x_{i}^{n}
\end{eqnarray*}
where $\alpha_{i}\in\mathbb{R}$. Then we have the following theorem.

\begin{theorem}[\textbf{Second main theorem for a homogeneous Fermat polynomial}]\label{theorem5.8}
    Let $m\in\mathbb{N}$ and $f:\mathbb{R}\rightarrow \mathbb{T}\mathbb{R}^{m}$ be a $1$-st tropical holomorphic curve with a reduced representation $\textbf{f}=(f_{0}, \cdots, f_{m})$ and $\mathcal{P}(x)=\sum_{i=0}^{m}\alpha_{i}x_{i}^{n}$ where $n\in\mathbb{N}$ and $\alpha_{i}>0$, $i=0, \cdots, m$. If
    \begin{eqnarray}\label{5f1}
        \limsup\limits_{r\rightarrow +\infty}\frac{r}{T_{f}(r)}=0,
    \end{eqnarray}
    then we have
    \begin{eqnarray}\label{5f2}
        \theta T_{f}^{n}(r)\leq \sum_{j=1}^{n}N^{(j)}\left(r, \frac{1_{0}}{\mathcal{P}\circ f}\oslash\right)+o(T_{f}^{n}(r))\leq \Theta T_{f}^{n}(r),
    \end{eqnarray}
    where $\theta=\min\limits_{0\leq i\leq m}\alpha_{i}$, $\Theta=2^{n-1}\sum_{i=0}^{m}\alpha_{i}$ and $r$ tends to infinity.
\end{theorem}
\begin{proof}
    For any linear function $g(x)=kx+b$, we can easily get 
    \begin{eqnarray*}
        \frac{(g^{n}(x))^{(j)}}{j!}=C_{n}^{j}g^{n-j}(x)k^{j},
    \end{eqnarray*}
    for all $j=0, \cdots, n$, $x\in\mathbb{R}$ where we assume $0^{0}=1$. Thus we can obtain
    \begin{eqnarray}\label{5a12}
        \omega_{f_{i}^{n}}^{(j)}(x)=C_{n}^{j}f_{i}^{n-j}(x)\left(sgn^{j+1}(x^{+})(f_{i}^{'}(x^{+}))^{j}-sgn^{j+1}(x^{-})(f_{i}^{'}(x^{-}))^{j}\right),
    \end{eqnarray}
    where $i=0, \cdots, m$ and $x\in\mathbb{R}$, from which we can see that
    \begin{eqnarray*}
        \{\text{the singularities of } f_{i}^{n}\}\subset\{\text{the singularities of } f_{i}\}\cup \{0\}.
    \end{eqnarray*}

    Since $f_{i}$ is $1$-st tropical entire, then it must satisfy one of the following two properties.

    (1) There exists a finite number $R_{i}>0$ such that $f_{i}^{'}(-r^{-})<0$ and $f_{i}^{'}(r^{+})>0$ when $r>R_{i}$;

    (2) $f_{i}^{'}(x^{\pm})\geq 0$ or $\leq 0$ holds for any $x\in\mathbb{R}$.

    If $f_{i}$ satisfies (1), then there exists a finite value $R_{i}^{'}\geq R_{i}$ such that $f_{i}(\pm r)>0$ for $r>R_{i}^{'}$, thus
    \begin{eqnarray*}
        M_{1, i}&:=&\max\left\{\max\limits_{x\in[-R_{i}^{'}, R_{i}^{'}]}\left\{|f_{i}(x)|\right\}, 1\right\}<+\infty,\\
        M_{2, i}&:=&\max\left\{\max\limits_{x\in[-R_{i}^{'}, R_{i}^{'}]}\left\{|f_{i}^{'}(x^{\pm})|\right\}, 1\right\}<+\infty.
    \end{eqnarray*}
    Then for any $x\in \mathbb{R}$, if $|x|> R_{i}^{'}$, we have $\omega_{f_{i}^{n}}^{(j)}(x)\geq 0$ from (\ref{5a12}), which implies that 
    \begin{eqnarray}\label{5a13}
      \sum_{j=1}^{n}N_{(-r, -R_{i}^{'})\cup (R_{i}^{'}, r)}^{(j)}(r, f_{i}^{n})=0.
    \end{eqnarray}
    
    If $|x|\leq R_{i}^{'}$, then we have $|\omega_{f_{i}^{n}}^{(j)}(x)|\leq C_{n}^{j}M_{1, i}^{n-j}2M_{2, i}^{j}\leq 2n!(M_{1, i}M_{2, i})^{n}=:M_{i}<\infty$ for all $j=1, \cdots, n$ from (\ref{5a12}). Suppose $\{x_{k}\}_{k=1}^{\nu_{i}}$ are the singularities of $f_{i}$ in $[-R_{i}^{'}, R_{i}^{'}]$, and since there is no finite accumulation point of $f_{i}$'s singularities, it follows that $\nu_{i}<+\infty$. Hence
    \begin{eqnarray}\label{5a14}
        \sum_{j=1}^{n}N_{[-R_{i}^{'}, R_{i}^{'}]}^{(j)}(r, f_{i}^{n})&\leq& \frac{1}{2}\sum_{j=1}^{n}\sum_{k=1}^{\nu_{i}}|\omega_{f_{i}^{n}}^{(j)}(x_{k})|(r-|x_{k}|)^{j}\nonumber\\
        &\leq& \frac{1}{2}n\nu_{i}M_{i}r^{n}=o(T_{f}^{n}(r)),
    \end{eqnarray}
    where we assume $r>\max\{R_{i}^{'}, 1\}$. The last equality follows by the assumption (\ref{5f1}).
    
    If $f_{i}$ satisfies (2), without loss of generality, we assume $f_{i}$ is non-constant and $f_{i}^{'}(x^{\pm})\geq 0$ for all $x\in\mathbb{R}$. For a small enough quantity $\epsilon$, we denote
    \begin{eqnarray*}
        I_{i}=:\{x:f_{i}(x+\epsilon)f_{i}(x-\epsilon)<0, x\in\mathbb{R}\}.
    \end{eqnarray*}
    Then there is at most one point in $I_{i}$. Let
    \begin{eqnarray*}
        x^{*}=\left\{\begin{array}{ll}
         x, & x\in I_{i},\\
         0, & I_{i}= \varnothing.
         \end{array}\right.
    \end{eqnarray*}
    Then $x^{*}$ is a finite number. If $x>\max\{x^{*}, 0\}$, then $f_{i}(x)>0$, and hence $\omega_{f_{i}^{n}}^{(j)}(x)\geq 0$ for $j=1, \cdots, 0$, which means that
    \begin{eqnarray}\label{5a15}
        \sum_{j=1}^{n}N_{(\max\{x^{*}, 0\}, r)}^{(j)}(r, f_{i}^{n})=0.
    \end{eqnarray}
    With the same calculation of (\ref{5a14}), we can also get 
    \begin{eqnarray}\label{5a16}
        \sum_{j=1}^{n}N_{[\min\{x^{*}, 0\}, \max\{x^{*}, 0\}]}^{(j)}(r, f_{i}^{n})=o(T_{f}^{n}(r)).
    \end{eqnarray}
    
    If $x<\min\{x^{*}, 0\}$, we assume $f_{i, x^{*-}}(x)=k_{i}^{*}x+b_{i}^{*}$ where $k_{i}^{*}, b_{i}^{*}\in\mathbb{R}$ are finite values and $k_{i}^{*}\geq 0$. Since $f_{i}$ is non-decreasing and convex, then $|f(x)-b_{i}^{*}|\leq |k_{i}^{*}x|$ for $x<\min\{x^{*}, 0\}$, which leads to
    \begin{eqnarray}\label{5a17}
        |f_{i}(x)|\leq k_{i}^{*}|x|+O(1),
    \end{eqnarray}
    where $x<\min\{x^{*}, 0\}$. For any $r>|x^{*}|$ and then $f(r)\geq 0$, we denote by $\{x_{-l}\}_{l=1}^{\mu_{i}}$ the roots of $f_{i}$ in $(-r, \min\{x^{*}, 0\})$ and define a new $1$-st tropical entire function
    \begin{eqnarray*}
        g_{i}(x):=\left\{\begin{array}{ll}
         k_{i}^{*}x+b_{i}^{*}, & x> \min\{x^{*}, 0\},\\
         f_{i}(x), & x\leq \min\{x^{*}, 0\}.
         \end{array}\right.
    \end{eqnarray*}
    Then $m(r, \pm g_{i})\leq k_{i}^{*}r+O(1)$ and $\{x_{-l}\}_{l=1}^{\mu_{i}}$ are all $g_{i}$'s roots in $(-r, r)$. By following the $1$-st Jensen formula, we have
    \begin{eqnarray*}
        N(r, 1_{0}\oslash g_{i})&=&\frac{1}{2}\sum_{l=1}^{\mu_{i}}(f_{i}^{'}(x_{-l}^{+})-f_{i}^{'}(x_{-l}^{-}))(r-|x_{-l}|)\\
        &=&T(r, 1_{0}\oslash g_{i})-m(r, 1_{0}\oslash g_{i})\\
        &\leq&T(r, g_{i})+m(r, 1_{0}\oslash g_{i})+O(1)\\
        &=&m(r, g_{i})+m(r, 1_{0}\oslash g_{i})+O(1)\\
        &\leq&2k_{i}^{*}r+O(1).
    \end{eqnarray*}
    Therefore, from (\ref{5a12}), we have 
    \begin{eqnarray}\label{5a18}
        &&\sum_{j=1}^{n}N_{(-r, \min\{x^{*}, 0\})}^{(j)}(r, f_{i}^{n})\nonumber\\
        &\leq&\sum_{j=1}^{n}\sum_{l=1}^{\mu_{i}}n!\left(k_{i}^{*}r+O(1)\right)^{n-j}\left((f_{i}^{'}(x_{-l}^{+}))^{j}-(f_{i}^{'}(x_{-l}^{-}))^{j}\right)(r-|x_{-l}|)^{j}\nonumber\\
        &\leq&\sum_{j=1}^{n}n!\left(k_{i}^{*}r+O(1)\right)^{n-j}\sum_{l=1}^{\mu_{i}}\left((f_{i}^{'}(x_{-l}^{+})-f_{i}^{'}(x_{-l}^{-}))(r-|x_{-l}|)n^{2}(k^{*}r)^{j-1}\right)\nonumber\\
        &\leq&\sum_{j=1}^{n}(n+2)!\left(k_{i}^{*}r+O(1)\right)^{n-1}\left(\sum_{l=1}^{\mu_{i}}(f_{i}^{'}(x_{-l}^{+})-f_{i}^{'}(x_{-l}^{-}))(r-|x_{-l}|)\right)\nonumber\\
        &\leq&\sum_{j=1}^{n}(n+2)!\left(k_{i}^{*}r+O(1)\right)^{n-1}\left(4k_{i}^{*}r+O(1)\right)\nonumber\\
        &\leq&4(n+3)!\left(k_{i}^{*}r+O(1)\right)^{n}\nonumber\\
        &=&o(T_{f}^{n}(r))
    \end{eqnarray}
    where we used $0\leq f_{i}^{'}(x_{-l}^{\pm})\leq k^{*}$ for $l=1, \cdots, \mu_{i}$ and also used a basic inequality $a^{j}-b^{j}=(a-b)(\sum_{s=0}^{j-1}a^{s}b^{j-s-1})\leq n(a-b)a^{j-1}$ for $0<b\leq a$ and $j=1, \cdots, n$ in the second inequality. We can also do a similar analysis in the case where $f_{i}^{'}(x^{\pm})\leq 0$ for all $x\in\mathbb{R}$.

    By combining equations (\ref{5a13}), (\ref{5a14}), (\ref{5a15}), (\ref{5a16}) and (\ref{5a18}), we obtain
    \begin{eqnarray}\label{5a19}
        \sum_{j=1}^{n}N^{(j)}(r, f_{i}^{n})=o(T_{f}^{n}(r)).
    \end{eqnarray}
    
    For any $n$-th tropical meromorphic functions $h_{1}$, $h_{2}$ and any positive constant $\alpha$, we can easily get that
    \begin{eqnarray*}
        \sum_{j=1}^{n}N^{(j)}(r, h_{1}+h_{2})&\leq& \sum_{j=1}^{n}N^{(j)}(r, h_{1})+\sum_{j=1}^{n}N^{(j)}(r, h_{2})\\
        \sum_{j=1}^{n}N^{(j)}(r, \alpha h_{1})&=&\alpha \sum_{j=1}^{n}N^{(j)}(r, h_{1}).
    \end{eqnarray*}
    Hence
    \begin{eqnarray*}
        \sum_{j=1}^{n}N^{(j)}(r, \mathcal{P}\circ f)\leq \sum_{i=0}^{m}\alpha_{i}\sum_{j=1}^{n}N^{(j)}(r, f_{i}^{n})=o(T_{f}^{n}(r)).
    \end{eqnarray*}

    Since $\mathcal{P}\circ f$ is a $n$-th tropical meromorphic function, then by applying the $n$-th Jensen formula, we have
    \begin{eqnarray}\label{Pn}
        &&\sum_{j=1}^{n}N^{(j)}\left(r, \frac{1_{0}}{\mathcal{P}\circ f}\oslash\right)\nonumber\\
        &=&m(r, \mathcal{P}\circ f)-m\left(r, \frac{1_{0}}{\mathcal{P}\circ f}\oslash\right)+\sum_{j=1}^{n}N^{(j)}\left(r, \mathcal{P}\circ f\right)+O(1)\nonumber\\
        &=&\frac{1}{2}\sum_{\delta=\pm 1}\mathcal{P}\circ f(\delta r)+o(T_{f}^{n}(r)).
    \end{eqnarray}

From (\ref{5a17}), when $r$ is sufficient large, $f_{i}$ satisfies one of the following 3 properties

(1) $f_{i}(\pm r)\geq 0$;

(2) $f_{i}(r)\geq 0$, $f_{i}(-r)<0$ and $|f_{i}(-r)|\leq kr+O(1)$ where $k$ is a finite positive constant;

(3) $f_{i}(-r)\geq 0$, $f_{i}(r)<0$ and $|f_{i}(r)|\leq kr+O(1)$ where $k$ is a finite positive constant.

    Let $F(x)=\max\limits_{0\leq i\leq m}\{f_{i}(x)\}$, then $F(\pm r)$ can not both be negative simultaneously when $r$ is large enough, because all $f_{i}$ are tropical entire and $f$ is not a constant holomorphic curve. Now we claim that
    \begin{eqnarray}\label{5a11}
        \left(\frac{F(r)+F(-r)}{2}\right)^{n}\leq \frac{F^{n}(r)+F^{n}(-r)}{2}+o(T_{f}^{n}(r))\leq \frac{(F(r)+F(-r))^{n}}{2}
    \end{eqnarray}
    as $r$ tends to infinity. 
    
    We first suppose $F(\pm r)>0$ for all large enough $r$. If $F(-\delta r)=o(F(\delta r))$ for some $\delta=\pm 1$, then (\ref{5a11}) becomes $(F(\delta r)/2)^{n}\leq F^{n}(\delta r)/2\leq F^{n}(\delta r)/2$ which is valid. If $F(-\delta r)= a F(\delta r)(1+o(1))$ for some constants $a\geq 0$ and $\delta=\pm 1$, without loss of generality, we assume $a\geq 1$, otherwise we can consider $F(\delta r)= \frac{1}{a}F(-\delta r)(1+o(1))$. Then (\ref{5a11}) can be rewritten as $\left(\frac{a+1}{2}\right)^{n}F^{n}(\delta r)\leq \frac{1+a^{n}}{2}F^{n}(\delta r)\leq \frac{(1+a)^{n}}{2}F^{n}(\delta r)$ which is valid as well; the first inequality follows by considering the derivative of $\eta(t)=(1+(2t-1)^{n})/2-t^{n}$ where $t=\frac{1+a}{2}\geq 1$ and $n\in\mathbb{N}$.

    If $F(-\delta r)<0$ for large enough $r$ and $\delta=1$ or $-1$, then from above properties of $f_{i}$, we know that $|F(-\delta r)|=o(T_{f}(r))$ and $F(\delta r)=2T_{f}(r)+o(T_{f}(r))$, which implies $F(-\delta r)=o(F(\delta r))$. Then the claim follows by a similar analysis as above.
    
    Hence, we complete the proof of the claim. By applying (\ref{5a11}), it follows, on one hand, that
    \begin{eqnarray}\label{Pn1}
        \frac{1}{2}\sum_{\delta=\pm 1}\mathcal{P}\circ f(\delta r)&=&\frac{1}{2}\sum_{\delta=\pm 1}\sum_{i=0}^{m}\alpha_{i}f_{i}^{n}(\delta r)\nonumber\\
        &\leq&\frac{1}{2}\sum_{\delta=\pm 1}\left(\sum_{i=0}^{m}\alpha_{i}\right)F^{n}(\delta r)+o(T_{f}^{n}(r))\nonumber\\
        &\leq&2^{n-1}\left(\sum_{i=0}^{m}\alpha_{i}\right)\left(\frac{F(r)+F(-r)}{2}\right)^{n}+o(T_{f}^{n}(r))\nonumber\\
        &=&\Theta T_{f}^{n}(r)+o(T_{f}^{n}(r)),
    \end{eqnarray}
    and, on the other hand, 
    \begin{eqnarray}\label{Pn2}
        \frac{1}{2}\sum_{\delta=\pm 1}\mathcal{P}\circ f(\delta r)&=&\frac{1}{2}\sum_{\delta=\pm 1}\sum_{i=0}^{m}\alpha_{i}f_{i}^{n}(\delta r)\nonumber\\
        &\geq&\frac{1}{2}\theta\sum_{\delta=\pm 1}F^{n}(\delta r)+o(T_{f}^{n}(r))\nonumber\\
        &\geq& \theta\left(\frac{F(r)+F(-r)}{2}\right)^{n}+o(T_{f}^{n}(r))\nonumber\\
        &=&\theta T_{f}^{n}(r)+o(T_{f}^{n}(r)).
    \end{eqnarray}
    Thus the assertion follows by combining (\ref{Pn}), (\ref{Pn1}) and (\ref{Pn2}).
\end{proof}

Now we give several examples to show that the growth assumption (\ref{5f1}) can not be dropped and the lower and upper bounds in (\ref{5f2}) are both sharp.

\begin{exa}
    Let $f=[0: 2x]$ and $\mathcal{P}(x)=x_{0}^{2}+2x_{1}^{2}$. Then $T_{f}(r)=r$ does not satisfy the growth assumption (\ref{5f1}). Since $\mathcal{P}\circ f=8x^{2}$, then $\sum_{j=1}^{2}N^{(j)}(r, 1_{0}\oslash (\mathcal{P}\circ f))=8r^{2}$ does not satisfy the inequality (\ref{5f2}) either.
\end{exa}

\begin{exa}
    Let 
    \begin{eqnarray*}
        f_{1}(x)=\left\{\begin{array}{ll}
         0, & x\leq 1,\\
         (2k-1)x-(2k^{2}-1), & x\in [2k-1, 2k+1),
         \end{array}\right.
    \end{eqnarray*}
    and
    \begin{eqnarray*}
         f_{2}(x)=\left\{\begin{array}{ll}
         0, & x\leq 2,\\
         2kx-2k(k+1), & x\in [2k, 2k+2),
         \end{array}\right.
    \end{eqnarray*}
    where $k\in\mathbb{N}.$
    
    (1) Let $g=[0: f_{1}(x)]$ and $\mathcal{P}_{1}(x)=x_{0}+x_{1}$, then we can easily get that $T_{g}(r)=N(r, 1_{0}\oslash (\mathcal{P}_{1}\circ g))=\frac{1}{4}r^{2}+o(r^{2})$, which shows the lower bound can be obtained.

    (2) Let $h=[f_{1}(x): f_{2}(x)]$ and $\mathcal{P}_{2}(x)=x_{0}+x_{1}$, then we can also get that $T_{h}(r)=\frac{1}{4}r^{2}+o(r^{2})$ and $N(r, 1_{0}\oslash (\mathcal{P}_{2}\circ h))=T(r, f_{1}+f_{2})+O(1)=m(r, f_{1}+f_{2})+O(1)=\frac{1}{2}r^{2}+o(r^{2})$, which shows the upper bound can be obtained.
\end{exa}

\section{A relationship about tropical truncated second main theorem}

In \cite{Pro}, the first author and Tohge considered an example of reduction in the term $\sum_{i=0}^{m}N(r, 1_{0}\oslash f_{i})-N(r, 1_{0}\oslash C_{0}(f_{0}, \cdots, f_{m}))$ where $f_{i}$ are classical tropical entire functions, but provided no general statement of the phenomenon of reduction. Now we first explore the case of non-constant tropical holomorphic curve $f$ with a reduced representation $\textbf{f}=(f_{0}, \dots, f_{m})$ satisfying the properties that all $f_{i}$ have finitely many roots, and we will follow the classical notations $\omega_{f}(x)=\omega_{f}^{(1)}(x)$ and $N(r, f)=N^{(1)}(r, f)$. For any fixed $i=0, \cdots, m$, we suppose $\{x_{i, t}\}_{t=1}^{\nu_{i}}$ are all the roots of $f_{i}$, where $\nu_{i}\in\mathbb{N}$, and let $x_{i}^{*}:=\max\limits_{1\leq t\leq \nu_{i}}\{|x_{i, t}|\}$. Then there exist finite constants $b_{i}^{-}, b_{i}^{+}, k_{i}^{-}, k_{i}^{+}\in\mathbb{R}$ and $k_{i}^{-}\leq k_{i}^{+}$ such that
\begin{eqnarray*}
         f_{i}(x)=\left\{\begin{array}{ll}
         k_{i}^{-}x+b_{i}^{-}, & x< -x_{i}^{*},\\
         k_{i}^{+}x+b_{i}^{+}, & x> x_{i}^{*}.
         \end{array}\right.
    \end{eqnarray*}
This yields $\sum_{t=1}^{\nu_{i}}\omega_{f_{i}}(x_{i, t})=k_{i}^{+}-k_{i}^{-}\geq 0$, and then by making use of the definition $N(r, f_{i})=\frac{1}{2}\sum_{t=1}^{\nu_{i}}\omega_{f_{i}}(x_{i, t})(r-|x_{i, t}|)$, we have
\begin{eqnarray*}
    \frac{k_{i}^{+}-k_{i}^{-}}{2}r+O(1)&=&\frac{1}{2}(r-|x_{i}^{*}|)\sum_{t=1}^{\nu_{i}}\omega_{f_{i}}(x_{i, t})\leq N(r, f_{i})\\
    &\leq& \frac{1}{2}r\sum_{t=1}^{\nu_{i}}\omega_{f_{i}}(x_{i, t})=\frac{k_{i}^{+}-k_{i}^{-}}{2}r,
\end{eqnarray*}
as $r$ tends to infinity. Hence
\begin{eqnarray}\label{6a111}
    \sum_{i=0}^{m}N(r, f_{i})=\left(\sum_{i=0}^{m}\frac{k_{i}^{+}-k_{i}^{-}}{2}\right)r+O(1),
\end{eqnarray}
where $\sum_{i=0}^{m}\frac{k_{i}^{+}-k_{i}^{-}}{2}>0$ because $f$ is non-constant. Moreover, let $\{\pi_{0}(x), \cdots, \pi_{m}(x)\}$ be a permutation of $\{0, \cdots, m\}$ depending on $x$ such that $C_{0}(f_{0}, \cdots, f_{m})(x)=\sum_{i=0}^{m}\overline{f}_{i}^{[\pi(i)(x)]}(x)$. Then there exist two finite constants $b^{-}$ and $b^{+}$ and a large enough $R>0$ such that  
\begin{eqnarray*}
         C_{0}(f_{0}, \cdots, f_{m})(x)=\left\{\begin{array}{ll}
         \sum_{i=0}^{m}k_{i}^{-}x+b^{-}, & x< -R,\\
         \sum_{i=0}^{m}k_{i}^{+}x+b^{+}, & x> R.
         \end{array}\right.
    \end{eqnarray*}
Then by following a similar analysis as with $N(r, f_{i})$ above, we obtain
\begin{eqnarray}\label{6a222}
    N(r, C_{0}(f_{0}, \cdots, f_{m}))=\frac{\sum_{i=0}^{m}k_{i}^{+}-\sum_{i=0}^{m}k_{i}^{-}}{2}r+O(1),
\end{eqnarray}
as $r$ tends to infinity. Thus, $N(r, C_{0}(f_{0}, \cdots, f_{m}))=(\sum_{i=0}^{m}N(r, f_{i}))(1+o(1))$ follows by combining (\ref{6a111}) and $(\ref{6a222})$.

For more general cases (we mean $f_{i}$ are allowed to have infinitely many roots), we give below a theorem to show the relationship with a growth assumption in the $n$-th setting, which is also the second part of the previous tropical second main theorem when $n=1$.

Before giving the relationship, let's recall \cite[Lemma 5.1]{Pro}. We have the same conclusion in our setting and the proof is also exactly same, so we omit the details here.
\begin{lemma}
    (\cite[Lemma 5.1]{Pro}) If $f_{0}, \cdots, f_{m}$ and $g$ are tropical entire functions, then
    \begin{eqnarray*}
        C_{0}(f_{0}\otimes g, \cdots, f_{m}\otimes g)=g\otimes \overline{g}\cdots\otimes \overline{g}^{[m]}\otimes C_{0}(f_{0}, \cdots, f_{m}).
    \end{eqnarray*}
\end{lemma}
We refer to \cite{Rbook} and \cite{Pro} for more details about the tropical Casoratian.

\begin{theorem}\label{theorem6.2}
    Let $n, m\in\mathbb{N}$ and $f:\mathbb{R}\rightarrow \mathbb{T}\mathbb{P}^{m}$ be a $n$-th tropical holomorphic curve in $\mathbb{T}\mathbb{P}^{m}$ with a reduced representation $\textbf{f}=(f_{0}, \cdots, f_{m})$. If there exists a constant $l\in\{0, \cdots, m\}$ such that $f_{i}\oslash f_{l}$ is well defined for all $i=0, \cdots, m$,
    \begin{eqnarray}\label{6f1}
        \limsup\limits_{r\rightarrow +\infty}\frac{\log\log T_{f}(r)}{\log r}:=\rho_{2}<1,
    \end{eqnarray}
    and $\sigma\in(0, 1-\rho_{2})$, then
    \begin{eqnarray}\label{6f2}
        &&\left|\sum_{i=0}^{m}\sum_{j=1}^{n}N^{(j)}\left(r, \frac{1_{0}}{f_{i}}\oslash\right)-\sum_{j=1}^{n}N^{(j)}\left(r, \frac{1_{0}}{C_{0}(f_{0}, \cdots, f_{m})}\oslash\right)\right|\nonumber\\
        &\leq&\sum_{j=1}^{n}\left|N^{(j)}(r, C_{0}(f_{0}, \cdots, f_{m}))-\sum_{i=0}^{m}N^{(j)}(r, \overline{f}_{i}^{[i]})\right|\nonumber\\
        &&+\sum_{i=0}^{m}\sum_{j=1}^{\lfloor \frac{n}{2}\rfloor}\left|N_{[-2i, 0]}^{(2j)}\left(r, \frac{1_{0}}{\overline{f}_{i}^{[i]}}\oslash\right)-N_{[-i, i]}^{(2j)}\left(r, \frac{1_{0}}{f_{i}}\oslash\right)\right|+o\left(\frac{T_{f}(r)}{r^{\sigma}}\right),
    \end{eqnarray}
    where $r$ approaches infinity outside of a set of finite logarithmic measure.
\end{theorem}

\begin{proof}
    We first claim that
    \begin{eqnarray}\label{equa1}
    &&\sum_{j=1}^{n}\left|N^{(j)}\left(r, \frac{1_{0}}{\overline{f}_{i}^{[k]}}\oslash\right)-N^{(j)}\left(r, \frac{1_{0}}{f_{i}}\oslash\right)\right|\nonumber\\
    &=& \sum_{j=1}^{\lfloor \frac{n}{2}\rfloor}\left|N_{[-2k, 0]}^{(2j)}\left(r, \frac{1_{0}}{\overline{f}_{i}^{[k]}}\oslash\right)-N_{[-k, k]}^{(2j)}\left(r, \frac{1_{0}}{f_{i}}\oslash\right)\right| +o\left(\frac{T_{f}(r)}{r^{\sigma}}\right)
    \end{eqnarray}
    for all $j=1, \cdots, n$ and $i, k=0, \cdots, m$ as $r$ tends to infinity outside of a set of finite logarithmic measure. For any fixed $i, j, k$ and $r>2k$, we define the set of all $j$-th roots of $f_{i}$ in $(-r-k, r+k)$ as
    \begin{eqnarray*}
        \{x_{1,s}\}_{s=1}^{\mu_{1}}&\subset& [-k, k], \\
        \{x_{2,s}\}_{s=1}^{\mu_{2}}&\subset& (-r+k, -k),\\
        \{x_{3,s}\}_{s=1}^{\mu_{3}}&\subset& (k, r-k),\\
        \{x_{4,s}\}_{s=1}^{\mu_{4}}&\subset& (-r, -r+k],\\
        \{x_{5,s}\}_{s=1}^{\mu_{5}}&\subset& [r-k, r),\\
        \{x_{6,s}\}_{s=1}^{\mu_{6}}&\subset& (-r-k, -r],\\
        \{x_{7,s}\}_{s=1}^{\mu_{7}}&\subset& [r, r+k).
    \end{eqnarray*}
    Since $\left(\overline{g}^{[k]}\right)^{(j)}(x-k)=\frac{d^{j}g(t+k)}{dt^{j}}|_{t=x-k}=g^{(j)}(t+k)|_{t=x-k}=g^{(j)}(x)$ for any polynomial $g$ and $x\in\mathbb{R}$, it follows that
    \begin{eqnarray*}
        \omega_{\overline{f}_{i}^{[k]}}^{(j)}(x-k)=\frac{sgn^{j+1}((x-k)^{+})f_{i}^{(j)}(x^{+})-sgn^{j+1}((x-k)^{-})f_{i}^{(j)}(x^{-})}{j!}
    \end{eqnarray*}
    where $x\in\mathbb{R}$. Thus $\omega_{\overline{f}_{i}^{[k]}}^{(j)}(x-k)=\omega_{f_{i}}^{(j)}(x)$ when $|x|>k$ or $j$ is an odd number, while this remains uncertain when $|x|\leq k$ and $j$ is an even number. Hence, we can define all $\overline{f}_{i}^{[k]}$'s $j$-th roots in $(-r-2k, r)$ as
    \begin{eqnarray*}
        \{x_{1,s}^{'}\}_{s=1}^{\mu_{1}^{'}}&\subset& [-2k, 0], \\
        \{x_{2,s}-k\}_{s=1}^{\mu_{2}}&\subset& (-r, -2k),\\
        \{x_{3,s}-k\}_{s=1}^{\mu_{3}}&\subset& (0, r-2k),\\
        \{x_{4,s}-k\}_{s=1}^{\mu_{4}}&\subset& (-r-k, -r],\\
        \{x_{5,s}-k\}_{s=1}^{\mu_{5}}&\subset& [r-2k, r-k),\\
        \{x_{6,s}-k\}_{s=1}^{\mu_{6}}&\subset& (-r-2k, -r-k],\\
        \{x_{7,s}-k\}_{s=1}^{\mu_{7}}&\subset& [r-k, r).
    \end{eqnarray*}
    $\{x_{1,s}\}$ and $\{x_{1,s}^{'}\}$ may not be one-to-one when $j$ is an even number, but $\{x_{t,s}\}$ and $\{x_{t,s}-k\}$ should be one-to-one and the identity $\omega_{f}^{(j)}(x_{t,s})=\omega_{\overline{f}_{i}^{[k]}}^{(j)}(x_{t,s}-k)$ holds for all $t=2, \cdots, 7$. Furthermore, if $j$ is an even number, we have
    \begin{eqnarray}
        N^{(j)}\left(r-k, \frac{1_{0}}{f_{i}}\oslash\right)
        &=&\sum_{s=1}^{\mu_{2}}\omega_{f}^{(j)}(x_{2,s})(r-k+x_{2, s})^{j}\nonumber\\
        &&+\sum_{s=1}^{\mu_{3}}\omega_{f}^{(j)}(x_{3,s})(r-k-x_{3,s})^{j}\nonumber\\
        &&+N_{[-k, k]}^{(j)}\left(r-k, \frac{1_{0}}{f_{i}}\oslash\right),\label{6aaa1}
        \end{eqnarray}
        \begin{eqnarray}
        N^{(j)}\left(r+k, \frac{1_{0}}{f_{i}}\oslash\right)&=&\sum_{t=1}^{3}\sum_{s=1}^{\mu_{2t}}\omega_{f}^{(j)}(x_{2t,s})(r+k+x_{2t,s})^{j}\nonumber\\
        &&+\sum_{t=1}^{3}\sum_{s=1}^{\mu_{2t+1}}\omega_{f}^{(j)}(x_{2t+1,s})(r+k-x_{2t+1, s})^{j}\nonumber\\
        &&+N_{[-k, k]}^{(j)}\left(r+k, \frac{1_{0}}{f_{i}}\oslash\right)\label{6aaa2},
        \end{eqnarray}
    and
    \begin{eqnarray}
        N^{(j)}\left(r, \frac{1_{0}}{\overline{f}_{i}^{[k]}}\oslash\right)&=&\sum_{s=1}^{\mu_{2}}\omega_{\overline{f}_{i}^{[k]}}^{(j)}(x_{2,s}-k)(r+x_{2,s}-k)^{j}\nonumber\\
        &&+\sum_{t=1}^{3}\sum_{s=1}^{\mu_{2t+1}}\omega_{\overline{f}_{i}^{[k]}}^{(j)}(x_{2t+1,s}-k)(r-x_{2t+1, s}+k)^{j}\nonumber\\
        &&+N_{[-2k, 0]}^{(j)}\left(r, \frac{1_{0}}{\overline{f}_{i}^{[k]}}\oslash\right).\label{6aaa3}
    \end{eqnarray}
    Since $k$ is a finite number and there is no finite accumulation point of $f_{i}$'s singularities, then we have
    \begin{eqnarray*}
        N_{[-k, k]}^{(j)}\left(r\pm k, \frac{1_{0}}{f_{i}}\oslash\right)&=&
        N_{[-k, k]}^{(j)}\left(r, \frac{1_{0}}{f_{i}}\oslash\right)(1+O(1/r))\\
        &=&N_{[-k, k]}^{(j)}\left(r, \frac{1_{0}}{f_{i}}\oslash\right)+o\left(\frac{T(r, f)}{r^{\sigma}}\right),
    \end{eqnarray*}
    as $r$ tends to infinity. By comparing the terms in (\ref{6aaa1}), (\ref{6aaa2}) and (\ref{6aaa3}), we obtain
    \begin{eqnarray}\label{NN1}
        &&N^{(j)}\left(r-k, \frac{1_{0}}{f_{i}}\oslash\right)\nonumber\\
        &\leq& N^{(j)}\left(r, \frac{1_{0}}{\overline{f}_{i}^{[k]}}\oslash\right)+N_{[-k, k]}^{(j)}\left(r, \frac{1_{0}}{f_{i}}\oslash\right)-N_{[-2k, 0]}^{(j)}\left(r, \frac{1_{0}}{\overline{f}_{i}^{[k]}}\oslash\right)+o\left(\frac{T(r, f)}{r^{\sigma}}\right)\nonumber\\
        &\leq& N^{(j)}\left(r+k, \frac{1_{0}}{f_{i}}\oslash\right).
    \end{eqnarray}
    When $j$ is odd, then $\{x_{1,s}\}$ and $\{x_{1,s}^{'}\}$ are also one-to-one and $\omega_{f}^{(j)}(x_{1,s})=\omega_{\overline{f}_{i}^{[k]}}^{(j)}(x_{1,s}^{'})$ in this case, thus
    \begin{eqnarray}\label{NN2}
        N^{(j)}\left(r-k, \frac{1_{0}}{f_{i}}\oslash\right)\leq N^{(j)}\left(r, \frac{1_{0}}{\overline{f}_{i}^{[k]}}\oslash\right)\leq N^{(j)}\left(r+k, \frac{1_{0}}{f_{i}}\oslash\right).
    \end{eqnarray}
    
    Because $f_{i}$ is $n$-th entire, then from the $n$-th Jensen formula,
    \begin{eqnarray*}
        \sum_{j=1}^{n}N^{(j)}\left(r, \frac{1_{0}}{f_{i}}\oslash\right)&=&m(r, f_{i})-m\left(r, \frac{1_{0}}{f_{i}}\oslash\right)+O(1)\\
        &=&\frac{1}{2}(f_{i}(r)+f_{i}(-r))+O(1)\\
        &\leq&T_{f}(r)+O(1).
    \end{eqnarray*}
    Thus $N^{(j)}\left(r, \frac{1_{0}}{f_{i}}\oslash\right)\leq T_{f}(r)+O(1)$ for all $j=1, \cdots, n$, and then (\ref{6f1}) implies that  
    \begin{eqnarray*}
        \limsup\limits_{r\rightarrow +\infty}\frac{\log\log N^{(j)}\left(r, \frac{1_{0}}{f_{i}}\oslash\right)}{\log r}:=\rho_{2, j}\leq \rho_{2}<1.
    \end{eqnarray*}
    Hence $\sigma\in(0, 1-\rho_{2, j})$, and then from Lemma \ref{lemma4.5}, we have
    \begin{eqnarray}\label{6aa1}
        N^{(j)}\left(r\pm k, \frac{1_{0}}{f_{i}}\oslash\right)&=&N^{(j)}\left(r, \frac{1_{0}}{f_{i}}\oslash\right)+o\left(\frac{N^{(j)}(r, 1_{0}\oslash f_{i})}{r^{\sigma}}\right)\nonumber\\
        &=&N^{(j)}\left(r, \frac{1_{0}}{f_{i}}\oslash\right)+o\left(\frac{T_{f}(r)}{r^{\sigma}}\right)
    \end{eqnarray}
    as $r$ tends to infinity outside of a set of finite logarithmic measure, which proves the claim by combing (\ref{NN1}), (\ref{NN2}), (\ref{6aa1}) and adding all $j$.

    Furthermore, for any $i=0, \cdots, m$, inequality (\ref{div}) implies
    \begin{eqnarray*}
        \limsup\limits_{r\rightarrow +\infty}\frac{\log\log T\left(r,\frac{f_{i}}{f_{l}}\oslash\right)}{\log r}:=\rho_{2}^{'}\leq \rho_{2}<1.
    \end{eqnarray*}
    This yields $\sigma\in(0, 1-\rho_{2}^{'})$. Since $f_{i}\oslash f_{l}$ is well defined, then by applying Theorem \ref{theorem 4.6} and inequality (\ref{div}) again, we have
    \begin{eqnarray*}
        &&\left|\left(\overline{\frac{f_{i}}{f_{l}}\oslash}\right)^{[k_{1}]}\oslash \left(\overline{\frac{f_{i}}{f_{l}}\oslash}\right)^{[k_{2}]}(\delta r)\right|\\&\leq&\left|\left(\overline{\frac{f_{i}}{f_{l}}\oslash}\right)^{[k_{1}]}\oslash \left(\frac{f_{i}}{f_{l}}\oslash\right)(\delta r)\right|+\left|\left(\overline{\frac{f_{i}}{f_{l}}\oslash}\right)^{[k_{2}]}\oslash \left(\frac{f_{i}}{f_{l}}\oslash\right)(\delta r)\right|\\
        &=& o\left(\frac{T(r, f_{i}\oslash f_{l})}{r^{\sigma}}\right)=o\left(\frac{T_{f}(r)}{r^{\sigma}}\right)
    \end{eqnarray*}
    for all $k_{1}, k_{2}\in\{0, \cdots, m\}$ and $\delta=\pm 1$ as $r$ tends to infinity outside of a set of finite logarithmic measure. Therefore
    \begin{eqnarray}\label{equa2}
        &&\left|\frac{1}{2}\sum_{\delta=\pm 1}\left(C_{0}(f_{0}, \cdots, f_{m})(\delta r)-\sum_{i=0}^{m}\overline{f}_{i}^{[i]}(\delta r)\right)\right|\nonumber\\
        &=&\frac{1}{2}\sum_{\delta=\pm 1}\left|C_{0}\left(\frac{f_{0}}{f_{l}}\oslash, \cdots ,\frac{f_{m}}{f_{l}}\oslash\right)(\delta r)-\left(\frac{f_{0}}{f_{l}}\oslash+\cdots+\frac{\overline{f}_{m}^{[m]}}{\overline{f}_{l}^{[m]}}\oslash\right) (\delta r)\right|\nonumber\\
        &\leq&\frac{1}{2}\sum_{\delta=\pm 1}\left|\sum_{i=0}^{m}\left(\overline{\frac{f_{i}}{f_{l}}\oslash}\right)^{[\pi_{i}(\delta r)]}\oslash \left(\overline{\frac{f_{i}}{f_{l}}\oslash}\right)^{[i]}
   (\delta r)\right|\nonumber\\
        &\leq&\frac{1}{2}\sum_{\delta=\pm 1}\sum_{i=0}^{m}\left|\left(\overline{\frac{f_{i}}{f_{l}}\oslash}\right)^{[\pi_{i}(\delta r)]}\oslash \left(\overline{\frac{f_{i}}{f_{l}}\oslash}\right)^{[i]}(\delta r)\right|=o\left(\frac{T_{f}(r)}{r^{\sigma}}\right)
    \end{eqnarray}
    as $r$ tends to infinity outside of a set of finite logarithmic measure, where $\{\pi_{0}(\delta r), \cdots, \pi_{m}(\delta r)\}$ is a permutation of $\{0, \cdots, m\}$ depending on $\delta r$ such that $C_{0}(f_{0}, \cdots, f_{m})(\delta r)=\sum_{i=0}^{m}\overline{f}_{i}^{[\pi_{i}(\delta r)]}(\delta r)$.

By applying the $n$-th Jensen Formula again, we have
\begin{eqnarray}\label{equa3}
        &&\left|\sum_{j=1}^{n}N^{(j)}\left(r, \frac{1_{0}}{C_{0}(f_{0}, \cdots, f_{m})}\oslash\right)-\sum_{i=0}^{m}\sum_{j=1}^{n}N^{(j)}\left(r, \frac{1_{0}}{f_{i}}\oslash\right)\right|\nonumber\\
        &\leq&\sum_{i=0}^{m}\sum_{j=1}^{n}\left|N^{(j)}\left(r, \frac{1_{0}}{\overline{f}_{i}^{[i]}}\oslash\right)-N^{(j)}\left(r, \frac{1_{0}}{f_{i}}\oslash\right)\right|\nonumber\\
        &&+\left|\sum_{j=1}^{n}N^{(j)}\left(r, \frac{1_{0}}{C_{0}(f_{0}, \cdots, f_{m})}\oslash\right)-\sum_{i=0}^{m}\sum_{j=1}^{n}N^{(j)}\left(r, \frac{1_{0}}{\overline{f}_{i}^{[i]}}\oslash\right)\right|\nonumber\\
        &\leq&\sum_{i=0}^{m}\sum_{j=1}^{n}\left|N^{(j)}\left(r, \frac{1_{0}}{\overline{f}_{i}^{[i]}}\oslash\right)-N^{(j)}\left(r, \frac{1_{0}}{f_{i}}\oslash\right)\right|\nonumber\\
        &&+\left|\sum_{j=1}^{n}N^{(j)}\left(r, C_{0}(f_{0}, \cdots, f_{m})\right)-\sum_{i=0}^{m}\sum_{j=1}^{n}N^{(j)}\left(r, \overline{f}_{i}^{[i]}\right)\right|\nonumber\\
        &&+\left|\frac{1}{2}\sum_{\delta=\pm 1}\left(C_{0}(f_{0}, \cdots, f_{m})(\delta r)-\sum_{i=0}^{m}\overline{f}_{i}^{[i]}(\delta r)\right)\right|+O(1).
    \end{eqnarray}
    Hence the assertion follows by combining (\ref{equa1}), (\ref{equa2}) and (\ref{equa3}).
\end{proof}

Unlike in the case $n=1$, when $n\geq 2$, even if $f_{0}, \cdots, f_{m}$ are all tropical entire, we can not guarantee that $\frac{1_{0}}{\overline{f}_{i}^{[k]}}\oslash$ and $C_{0}(f_{0}, \cdots, f_{m})$ are tropical entire. Hence the right hand side of (\ref{6f2}) can not be simplified as $o(T_{f}(r)/r^{\sigma})$ in the general $n$-th setting, as we can see from the following example.
\begin{exa}
    Define a $2$-nd tropical holomorphic well defined curve $\textbf{f}=[f_{0}: f_{1}]$ where
    \begin{eqnarray*}
        f_{0}(x)=\left\{\begin{array}{ll}
         -x^{2}, & x\leq 0,\\
         x^{2}, & x>0,
         \end{array}\right. \qquad f_{1}(x)=\left\{\begin{array}{ll}
         x^{2}, & x\leq 0,\\
         -x^{2}, & x>0.
         \end{array}\right.
    \end{eqnarray*}
    Then
    \begin{eqnarray*}
        \overline{f}_{0}(x)=\left\{\begin{array}{ll}
         -(x+1)^{2}, & x\leq -1,\\
         (x+1)^{2}, & x>-1,
         \end{array}\right. \qquad \overline{f}_{1}(x)=\left\{\begin{array}{ll}
         (x+1)^{2}, & x\leq -1,\\
         -(x+1)^{2}, & x>-1,
         \end{array}\right.
    \end{eqnarray*}
    and
    \begin{eqnarray*}
        C_{0}(f_{0}, f_{1})=\left\{\begin{array}{ll}
         -2x-1, & x\leq -1,\\
         2x^{2}+2x+1, & -1<x\leq 0,\\
         2x+1, & 0<x.
         \end{array}\right.
    \end{eqnarray*}
    By a simple calculation, we have $T_{f}(r)=r^{2}$, $\sum_{i=0}^{2}\sum_{j=1}^{2}N^{(j)}\left(r, \frac{1_{0}}{f_{i}}\oslash\right)=0$ and $\sum_{j=1}^{2}N^{(j)}(r, 1_{0}\oslash C_{0}(f_{0}, f_{1}))=r^{2}$ which do not satisfy the inequality (\ref{6f2}) if we replace the right hand side of (\ref{6f2}) by $o(T_{f}(r)/r^{\sigma})$.
\end{exa}

When $n=1$, it is clear that $f_{l_{1}}\oslash f_{l_{2}}$ is well defined for all $l_{1}, l_{2}=0,\cdots, m$ and the right hand side of (\ref{6f2}) is $o(T_{f}(r)\slash r^{\sigma})$, and then we have the usual relationship put forward in \cite{Pro}.

\begin{corollary}\label{cor6.3}
    Let $m\in\mathbb{N}$ and $f:\mathbb{R}\rightarrow \mathbb{T}\mathbb{R}^{m}$ be a $1$-st tropical holomorphic curve with a reduced representation $\textbf{f}=(f_{0}, \cdots, f_{m})$. If
    \begin{eqnarray}\label{6f3}
        \limsup\limits_{r\rightarrow +\infty}\frac{\log\log T_{f}(r)}{\log r}:=\rho_{2}<1,
    \end{eqnarray}
    and $\sigma\in(0, 1-\rho_{2})$, then
    \begin{eqnarray}\label{6f4}
        \sum_{i=0}^{m}N\left(r, \frac{1_{0}}{f_{i}}\oslash\right)=N\left(r, \frac{1_{0}}{C_{0}(f_{0}, \cdots, f_{m})}\oslash\right)+o\left(\frac{T_{f}(r)}{r^{\sigma}}\right)
    \end{eqnarray}
    where $r$ approaches infinity outside an exceptional set of finite logarithmic measure.
\end{corollary}

The growth condition can not be dropped, as we can see from the example that $f=[0:  e_{\alpha}(x)]$ for some $\alpha>1$, thus $T_{f}(r)=T(r, e_{\alpha})+O(1)$ does not satisfy the assumption (\ref{6f3}) from (ii) Proposition \ref{pro3.8}. Since $\overline{e}_{\alpha}(x)=\alpha e_{\alpha}(x)$, then $C_{0}(f_{0}, f_{1})=\alpha e_{\alpha}(x)$. Thus $N(r, 1_{0}\oslash f_{0})+N(r, 1_{0}\oslash f_{1})=T(r, e_{\alpha})(1+o(1))$ while $N(r, 1_{0}\oslash C_{0}(f_{0}, f_{1}))=\alpha T(r, e_{\alpha})(1+o(1))$, which means the equality (\ref{6f4}) can not hold in the case.

From the Corollary \ref{cor6.3}, we can see that if we only consider the concept of truncation for the multiplicity by means of the shift operator, there is in practice no room left for any meaningful ramification, and hence there may not be a natural truncated second theorem in the tropical setting in this context.

\end{document}